\let\ORIlabel\label
\let\ORIrefstepcounter\refstepcounter
\AddToHook{package/hyperref/before}{\let\label\ORIlabel\let\refstepcounter\ORIrefstepcounter}

\PassOptionsToPackage{usenames,dvipsnames}{xcolor}
\documentclass[onefignum,onetabnum,twoside]{siamart220329}

\usepackage[T1]{fontenc}
\usepackage[utf8]{inputenc}
\usepackage[american]{babel}

\usepackage[sort]{cite}

\usepackage{algorithm}
\usepackage{algpseudocode}
\usepackage{amsmath}
\usepackage{amssymb}
\usepackage{booktabs}
\usepackage{hyperref}
\usepackage[shortlabels]{enumitem}
\usepackage{subfig}

\usepackage{bbm}
\usepackage{mathbbol}

\newsiamremark{remark}{Remark}
\newsiamremark{example}{Example}
\newsiamremark{assumption}{Assumption}

\usepackage{siunitx}
\usepackage{macros}

\usepackage[normalem]{ulem}
\usepackage{cancel}

\newcommand{\ie}{i.\,e.}

\headers%
  {Hyperbolic conservation laws}%
  {Kronbichler, Maier, Tomas}

\title{%
  Graph-based methods for hyperbolic systems of conservation laws using
  discontinuous space discretizations.}

\author{%
  Martin Kronbichler%
  \thanks{University of Augsburg, Germany and Ruhr University Bochum,
    Universitätsstraße 150, 44801 Bochum, Germany
    (\email{martin.kronbichler@rub.de})}
  \and%
  Matthias Maier%
  \thanks{Department of Mathematics, Texas A\&M University, 3368 TAMU,
    College Station, TX 77843, USA (\email{maier@tamu.edu})}
  \and%
  Ignacio Tomas%
  \thanks{Department of Mathematics and Statistics, Texas Tech University,
    2500 Broadway, Lubbock, TX 79409, USA (\email{igtomas@ttu.edu})
  }}

\begin{document}

\maketitle

\begin{abstract}
  We present a graph-based numerical method for solving hyperbolic systems
  of conservation laws using discontinuous finite elements. This work fills
  important gaps in the theory as well as practice of graph-based schemes.
  In particular, four building blocks required for the implementation  of
  flux-limited graph-based methods are developed and tested: a first-order
  method with mathematical guarantees of robustness; a high-order method based
  on the entropy viscosity technique; a procedure to compute local bounds; and
  a convex limiting scheme. Two important features of the current work are the
  fact that (i) boundary conditions are incorporated into the mathematical
  theory as well as the implementation of the scheme. For instance, the
  first-order version of the scheme satisfies pointwise entropy inequalities
  including boundary effects for any boundary data that is admissible; (ii)
  sub-cell limiting is built into the convex limiting framework. This is in
  contrast to the majority of the existing methodologies that consider a
  single limiter per cell providing no sub-cell limiting capabilities.

  From a practical point of view, the implementation of graph-based methods
  is algebraic, meaning that they operate directly on the stencil of the
  spatial discretization. In principle, these methods do not need to use or
  invoke loops on cells or faces of the mesh. Finally, we verify
  convergence rates on various well-known test problems with differing
  regularity. We propose a simple test in order to verify the
  implementation of boundary conditions and their convergence rates.
\end{abstract}



%%%%%%%%%%%%%%%%%%%%%%%%%%%%%%%%%%%%%%%%%%%%%%%%%%%%%%%%%%%%%%%%%%%%%%%%%%%%
%%%%%%%%%%%%%%%%%%%%%%%%%%%%%%%%%%%%%%%%%%%%%%%%%%%%%%%%%%%%%%%%%%%%%%%%%%%%
%%%%%%%%%%%%%%%%%%%%%%%%%%%%%%%%%%%%%%%%%%%%%%%%%%%%%%%%%%%%%%%%%%%%%%%%%%%%

\section{Introduction}

For the last four decades the field of numerical methods for solving
hyperbolic systems of conservation equations has been dominated by a
paradigm that is commonly referred to as \emph{high-resolution schemes}. These
are numerical methods in which the order of consistency is automatically
adjusted locally (in space) depending on some chosen smoothness criteria;
see early references \cite{Boris1997, Harten1972, Zale1979, Sweby1984}. While a
heuristic high-resolution method is a good starting point for practical
computations, it is not enough to achieve \emph{unconditional robustness}
of the scheme. Here, we define unconditional robustness as the guarantee that
the computed update at a given time step remains admissible and maintains
crucial physical invariants, such that the resulting state can be used again as
input for the next time step update.
Modern approaches for constructing \emph{robust} high-order schemes are
based on the following ingredients~\cite{Zale1979, Osti2022, Maier2023,
Pazner2021, Rueda2022, Chan2023, Kuzmin2020I, Kuzmin2020II, Shu1987,
Cockburn1989, Tez2006, Persson2006, Huerta2012, Zingan2013, Michoski2016,
Euler2018, DiscIndep2019}: (a) a reference low-order method with
mathematically guarantees of robustness; (b) a formally high-order
method that may or may not guarantee any robustness properties; and (c) a
post-processing procedure based on either flux or slope limiting techniques
that blends the low-order and high-order solutions. A particular incarnation of
such postprocessing technique is the \emph{convex limiting} technique that
establishes mathematical guarantees for maintaining a
(local) \emph{invariant-set} property \cite{Euler2018}.

First-order graph-based formulations combined with discontinuous spatial
discretizations are only tersely described in \cite[Section
4.3]{DiscIndep2019}, leaving the path towards high-performance high-order
graph-based methods underdeveloped. Therefore, the first goal of the
present paper is to complete the mathematical theory and discuss
computational aspects of invariant-set preserving schemes for the case of
discontinuous finite elements in comprehensive detail. In particular, we
incorporate boundary conditions into the formulation of the scheme, provide
proofs of invariant-set preservation and discrete entropy inequalities
including the effects of \emph{boundary data}.

The second goal of this paper is to lay out the elementary building blocks
required to construct a robust high-order scheme with convex limiting, using a
graph-based discontinuous Galerkin discretization. This requires the development
and testing of three components: (i) a heuristic high-order method, (ii) a
suitable strategy for constructing local bounds and their relaxation in order to
prevent degradation to first-order accuracy, and (iii) a convex-limiting
procedure that blends the high and low order methods while maintaining an
invariant set.

%%%%%%%%%%%%%%%%%%%%%%%%%%%%%%%%%%%%%%%%%%%%%%%%%%%%%%%%%%%%%%%%%%%%%%%%%%%%%%%%

\subsection{Background: graph-based methods}
A \emph{graph-based formulation}~\cite{Fletcher1983, Selmin1993, Lohner2008} is
a numerical method that operates directly on the \emph{stencil} or
\emph{sparsity-graph} of the discretization and its degrees of freedom,
bypassing entities such as cells, faces, or bilinear forms of the underlying
discretization paradigm. In its simplest incarnation, a graph-based formulation
takes a solution vector consisting of states $\statevector_i^n$ associated with
a (collocated) degree of freedom $i$ at time $t_n$ and computes an updated state
$\statevector_i^{n+1}$ for the time $t_{n+1}$ as follows:
\begin{align}
  \label{eq:graph_formulation}
  m_i \frac{\statevector_i^{n+1} - \statevector_i^n}{\dt_n} + \sum_{j \in
  \mathcal{I}(i)}
  \flux(\statevector_j^n)\bv{c}_{ij} - d_{ij} (\statevector_j^n -
  \statevector_i^n) = \bzero.
\end{align}
Here, $\mathcal{I}(i)$ is the stencil or ``adjacency list'' of the $i$th
degree of freedom, the expression $\sum_{j \in \mathcal{I}(i)}
\flux(\statevector_j^n)\bv{c}_{ij}$ is an algebraic representation of the
inviscid divergence operator, and $d_{ij} (\statevector_j^n -
\statevector_i^n)$ represents artificial viscous fluxes. The graph-based
formulation used in this paper is introduced in detail in
Section~\ref{sec:LowOrder}.

The concept of graph-based methods is quite old---one can argue that its roots
lie in finite difference approximations on unstructured grids. A notable modern
predecessor of what we call graph-based methods is the \emph{group finite
element formulation} of Fletcher \cite{Fletcher1983, Selmin1993}. More recently,
the concept of graph-based methods has been associated with flux-corrected
transport techniques; see \cite{FCTbook2005}. There is a rich record of
applications of graph-based (sometimes also called edge-based) methods in
computational fluid dynamics; see \cite[Chapter 10]{Lohner2008} for a historical
account. From the mathematical point of view, invariant-set preserving
methods~\cite{GuerPop2016} and convex-limiting \cite{Euler2018,
DiscIndep2019} techniques lead to increased interest in graph-based
formulations. Computationally, the algebraic structure of the scheme is a
natural idea in order to manipulate individual degrees of freedom and preserve
pointwise stability properties.

%%%%%%%%%%%%%%%%%%%%%%%%%%%%%%%%%%%%%%%%%%%%%%%%%%%%%%%%%%%%%%%%%%%%%%%%%%%%%%%%

\subsection{Background: invariant set preservation and convex limiting}
As outlined above, one of the fundamental prerequisites for the development
of a robust high-resolution scheme is the availability of a first-order
scheme with mathematically guaranteed robustness properties. The graph-based
methods described in \cite{GuerPop2016, DiscIndep2019} provide such
mathematical assurances, by maintaining the so-called \emph{invariant-set}
property \cite{GuerPop2016}. These first-order methods provide a discretization
agnostic paradigm that works on arbitrary meshes, arbitrary polynomial degree,
and arbitrary space dimension.

The \emph{convex-limiting framework} introduced in~\cite{Euler2018,
DiscIndep2019} provides a limiter technique that works on individual degrees
of freedom and maintains the invariant-set property. By construction, the
limiter formulation is purely algebraic and is not bound to a specific
discretization technique. In the context of finite element discretizations, it
does not distinguish between cell-interior degrees of freedom and those
located at the cell boundary. Therefore, convex limiting is an instance of
sub-cell flux limiting. First-order invariant-set preserving
methods, as well as the convex-limiting technique, were particularly well
received by the \emph{discontinuous Galerkin spectral element}
community~\cite{Pazner2021, Rueda2022, Chan2023} and related sub-cell
limiting efforts~\cite{Kuzmin2020I, Kuzmin2020II}.

%%%%%%%%%%%%%%%%%%%%%%%%%%%%%%%%%%%%%%%%%%%%%%%%%%%%%%%%%%%%%%%%%%%%%%%%%%%%%%%%
\subsection{Graph-based discontinuous Galerkin formulations and objectives}
The framework of invariant-set preserving methods based on convex limiting
have been discussed in detail in the context of \emph{continuous} finite
element formulations~\cite{ryujin-2021-1, Guermond2022, Clayton2023, Tovar2022}.
The low-order stencil-based invariant-set preserving methods have also been
combined effectively with discontinuous Galerkin spectral element methods
in \cite{Pazner2021, Rueda2022, Chan2023}. This requires combining two
different discretization techniques: the low-order method is described and
implemented using a purely algebraic (or \emph{stencil based})
approach, while the high-order method is most frequently described and
implemented with a cell-based formulation. Unifying the high and low order
methods into a single stencil-based description is very desirable, since it
greatly simplifies the analysis, construction, and implementation of the
schemes.

Therefore, we introduce and discuss suitable low-order and high-order
methods based on the algebraic structure introduced in \cite[Section
4.3]{DiscIndep2019}, using the same stencil, and discuss the convex
limiting paradigm adapted to the graph-based discontinuous Galerkin
setting. We will corroborate our analytical formulation with a
computational validation of convergence rates and qualitative fidelity.

%%%%%%%%%%%%%%%%%%%%%%%%%%%%%%%%%%%%%%%%%%%%%%%%%%%%%%%%%%%%%%%%%%%%%%%%%%%%%%%%
\subsection{Related works}
Regarding other methods, we start by noting that the dominant body of DG schemes
are cell-based, formally high-order methods, potentially supplemented with
slope-limiting (e.g. Zhang--Shu limiter \cite{Zhang2010}) or troubled-cell
indicator. These technique do not use a reference first-order scheme that
preserves all entropy inequalities and all invariant sets. While this may look
advantageous, since they do not compute both a
high-order and a low-order solution, they do not have local bounds
to enforce to the high-order method. That explains why the current
slope-limiting paradigm is mostly limited to positivity preservation ($\rho
> 0$ and $e > 0$) and rarely even uses local bounds. A noteworthy exceptions of
this approach are the publications \cite{Pazner2021, Rueda2022} which borrow
some ideas from \cite{DiscIndep2019}. These publications indeed rely on a
low-order finite volume subgrid method, while the high-order method is a
summation-by-parts method. The work in \cite{Pazner2021, Rueda2022} uses
high-order and low method with different spatial discretizations, thus
having different stencils and different basis functions. One of goals of the
present work is to present a simple approach that incorporates some form of
sub-cell limiting using a single spatial discretization.

%%%%%%%%%%%%%%%%%%%%%%%%%%%%%%%%%%%%%%%%%%%%%%%%%%%%%%%%%%%%%%%%%%%%%%%%%%%%%%%%
\subsection{Paper organization}
Section \ref{sec:prelim} is dedicated to preliminaries and notation about
the spatial discretization and hyperbolic systems of conservation laws. In
Section \ref{sec:LowOrder} we define a first-order method by extending the
mathematical description and analysis of graph-based discontinuous Galerkin
methods in terms of invariant-set preservation and entropy inequalities. In
Section \ref{sec:boundary_conditions} we discuss the procedures for
reflecting boundaries, supersonic inflows and outflows as well as subsonic
inflows and outflows. In Section \ref{sec:highorder} we develop a
convex-limited scheme using a robust first-order method and a high-order
method using the entropy-viscosity technique. In Section
\ref{CompExperiments} we provide a series of numerical experiments
demonstrating that the convex-limited scheme exhibits expected convergence
rates. We also propose a simple test for validating the implementation of
outflow and inflow boundary conditions.

%%%%%%%%%%%%%%%%%%%%%%%%%%%%%%%%%%%%%%%%%%%%%%%%%%%%%%%%%%%%%%%%%%%%%%%%%%%%%%%%
%%%%%%%%%%%%%%%%%%%%%%%%%%%%%%%%%%%%%%%%%%%%%%%%%%%%%%%%%%%%%%%%%%%%%%%%%%%%%%%%
%%%%%%%%%%%%%%%%%%%%%%%%%%%%%%%%%%%%%%%%%%%%%%%%%%%%%%%%%%%%%%%%%%%%%%%%%%%%%%%%
%%%%%%%%%%%%%%%%%%%%%%%%%%%%%%%%%%%%%%%%%%%%%%%%%%%%%%%%%%%%%%%%%%%%%%%%%%%%%%%%

\section{Preliminaries}\label{sec:prelim}

We briefly introduce relevant notation, recall the concepts of hyperbolic
conservation laws, invariant sets and entropy inequalities from the literature
and discuss the finite element setting for the proposed graph-based
discontinuous formulations. We loosely follow the notation from
\cite{Euler2018, DiscIndep2019}.

%%%%%%%%%%%%%%%%%%%%%%%%%%%%%%%%%%%%%%%%%%%%%%%%%%%%%%%%%%%%%%%%%%%%%%%%%%%%%%%%
\subsection{Hyperbolic systems, invariant sets, and entropy inequalities}
We are interested in solving partial differential equations of the form
\begin{align}
  \label{AbstractPDE}
  \partial_t \state + \diver \flux(\state) = \bzero,
\end{align}
where $\state = \state(\xcoord,t) \in \mathbb{R}^m$ is the state (here, $m$
is the number of components of the system), $\xcoord \in \mathbb{R}^d$ is
the spatial coordinate with the space dimension $d$, and $\flux(\state):
\mathbb{R}^m \rightarrow \mathbb{R}^{m \times d}$ denotes the flux. The
divergence of the flux is defined as $[\diver \flux(\state)]_i = \sum_{j
  \in \{1:d\}} \partial_{\xcoord_j} [\flux(\state)]_{ij}$. We make the following
assumptions, see also \cite{Viscous2014, GuerPop2016}.
\begin{assumption}[admissible set]
  \label{ass:admissible_set}
  We assume that there is a convex set $\mathcal{A} \subset \mathbb{R}^m$,
  called \emph{admissible set}, such that the matrix
  \begin{align*}
    \partial_x [\flux(\state)\normal] \in \mathbb{R}^{m \times m}
    \quad \text{with} \quad x := \normal \cdot \xcoord
  \end{align*}
  has real eigenvalues for all $\normal \in \mathbb{S}^{d-1}$. We assume that
  the solution of \eqref{AbstractPDE} is understood as the zero-viscosity
  limit $\state = \lim_{\varepsilon \rightarrow 0^+} \state^{\epsilon}$,
  where $\state^{\epsilon}$ solves the parabolic regularization
  \begin{align}
    \label{ParabReg}
    \partial_t \state^{\epsilon} + \diver \flux(\state^{\epsilon}) =
    \epsilon \Delta \state^{\epsilon}.
  \end{align}
\end{assumption}
\begin{assumption}[entropy inequality]
  \label{ass:entropy_inequality}
  Furthermore, we make the important assumption~\cite{Viscous2014,
  GuerPop2016} that there exists at least one entropy-flux pair
  $\{\entropy(\state), \eflux(\state)\}$ associated to \eqref{ParabReg},
  with $\entropy(\state):\mathbb{R}^m \rightarrow \mathbb{R}$ and
  $\eflux(\state):\mathbb{R}^m \rightarrow \mathbb{R}^d$ such that
  \begin{align*}
    \partial_t \entropy(\state^{\epsilon}) + \diver \eflux(\state^{\epsilon}) \leq
    \epsilon \Delta \entropy(\state^{\epsilon}) \ \ \text{for all} \ \epsilon > 0,
  \end{align*}
  see \cite[p. 28]{Godlewski1991}, such that the zero-viscosity limit
  $\state = \lim_{\varepsilon \rightarrow 0^+} \state^{\epsilon}$ satisfies the
  entropy dissipation inequality
  \begin{align*}
    \partial_t \entropy(\state) + \diver \eflux(\state) \leq 0 \, .
  \end{align*}
\end{assumption}
\begin{assumption}[Finite speed of propagation]
  \label{ass:propagation_speed}
  Consider the solution $\state(\xcoord,t)$ of the projected Riemann
  problem
  \begin{align}
  \label{eq:riemann_problem}
    \partial_{t}\state + \partial_x (\flux(\state) \cdot \normal) = 0,
    \quad\text{where}\quad
    \statevector_0 =
    \begin{cases}
      \state_{L} &\text{if } x \leq 0, \\
      \state_{R} &\text{if } x > 0, \\
    \end{cases}
  \end{align}
  and $x := \xcoord \cdot \normal$. Given the solution
  $\state(\xcoord,t)$ of the Riemann problem \eqref{eq:riemann_problem}, we
  assume that there is a maximum wavespeed of propagation, denoted as
  $\lambda_{\max}(\state_{L}, \state_{R}, \normal) > 0$, such that
  \begin{align*}
    \state(\xcoord,t) = \state_{L} \text{ for } \xcoord\cdot\normal \leq
    \tfrac{1}{2}\ \ \text{and} \ \
    \state(\xcoord,t) = \state_{R} \text{ for } \xcoord\cdot\normal \geq
    \tfrac{1}{2} \ \
  \end{align*}
  provided that $t \lambda_{\max}(\state_{L}, \state_{R}, \normal) \leq
  \tfrac{1}{2}$.
\end{assumption}
\begin{assumption}[Invariant set]
  \label{ass:invariant_set}
  We assume that the Riemann problem \eqref{eq:riemann_problem}
  satisfies an invariant-set property of the
  form:
  There exist convex subsets $\mathcal{B}\subset\mathcal{A}$ such that
  $\overline{\state}(t) := \int_{-1/2}^{1/2}\state(x,t)\,\text{d}x\, \in
  \mathcal{B}$, provided that $\state(x,t)$ is the unique entropy solution
  of the Riemann problem with left state $\state_{L}\in\mathcal{B}$ and
  right state $\state_{R}\in\mathcal{B}$ and that $t
  \lambda_{\max}(\state_{L}, \state_{R}, \normal) \leq \tfrac{1}{2}$.
  The precise description of the set $\mathcal{B} \subset \mathbb{R}^m$
  depends on the initial data and the hyperbolic system at hand; see
  Section~\ref{subse:euler_equations} for a detailed summary for the
  compressible Euler equations.
\end{assumption}
For additional background on entropy inequalities, parabolic
regularization principles and invariant sets we refer the reader to
\cite{Chueh1977, Hoff1985, Frid2001, Viscous2014, GuerPop2016} and
references therein. A general reference on hyperbolic systems of
conservation laws is \cite{Godlewski1991}.

%%%%%%%%%%%%%%%%%%%%%%%%%%%%%%%%%%%%%%%%%%%%%%%%%%%%%%%%%%%%%%%%%%%%%%%%%%%%%%%%
\subsection{The compressible Euler equations}
\label{subse:euler_equations}
For the compressible Euler equations~\cite{Toro2009,Godlewski1991} the
conserved state is $\state = [\rho, \mom, \totme]^\transp \in
\mathbb{R}^{m}$ with $m = d + 2$ is comprised of the density $\rho$, momentum
$\mom$, and total energy $\totme$. The flux $\flux(\state):\mathbb{R}^{d+2}
\rightarrow \mathbb{R}^{(d+2) \times d}$ is given by
\begin{align}
  \label{eq:euler_flux}
  \flux(\state) :=
  \begin{bmatrix}
    \mom \;,\; \rho^{-1} \mom \otimes \mom + p\,\identity\;,\;
\frac{\mom}{\rho}(\totme+p)
  \end{bmatrix}^T,
\end{align}
where $\identity$ is the $d\times d$ identity matrix, and $p$ is the
pressure. For the sake of simplicity we assume that the system is closed
with a polytropic ideal gas equation of state~\cite{Toro2009,Godlewski1991}.
This implies that
\begin{align}\label{PressurEuler}
  p \;=\; (\gamma -1)\,\varepsilon(\state),
  \quad
  \varepsilon(\state)\;:=\;
  \totme - \frac{1}{2}\frac{|\mom|^2}{\rho},
\end{align}
where $\varepsilon(\state)$ is the internal energy and $\gamma>1$ denotes
the ratio of specific heats. The admissible set $\mathcal{A}$ is given by
\begin{align}\label{AddSet}
  \mathcal{A} = \big\{ [\rho, \mom, \totme]^\transp \in \mathbb{R}^{d+2}
  \, | \, \rho > 0 \ \text{and} \ \varepsilon(\state) > 0 \big\}.
\end{align}
For the case of a polytropic ideal gas equation of state, the Euler
equations admit a \emph{mathematical} entropy-flux pair
$\big\{\eta(\state),\eflux(\state)\big\}$ given by
\begin{gather}
  \label{EntropyEuler}
  \eta(\state) := - \rho s(\state), \quad
  \eflux(\state) := - \mom s(\state), \quad
  \text{where}\; s(\state) := \log (\rho^{-\gamma}p(\state)).
\end{gather}
Here, $s(\state)$ denotes the \emph{specific} entropy. We note that this
choice of entropy-flux pair $\{\eta, \eflux\}$ is not unique, as there are
infinitely many entropy-flux pairs for the Euler equations
\cite{Haten1998, Hughes1986}.

Let $\state(\xcoord, t)$ be the unique entropy solution of the Riemann
problem \eqref{eq:riemann_problem}. Then, the Riemann average
$\overline{\state}(t)$ belongs to the invariant set
\begin{align*}
  \mathcal{B} &= \big\{ \state \in \mathcal{A} \, | \, s(\state) \geq
  \min\{s(\state_L), s(\state_R)\} \big\},
\end{align*}
provided that the initial data is admissible; \ie, $\state_L, \state_R \in
\mathcal{A}$; and that $t \lambda_{\max} \leq \tfrac{1}{2}$; see
Assumption~\ref{ass:propagation_speed}. Note that $\mathcal{B}$ characterizes a
minimum principle of the specific entropy $s$.

%%%%%%%%%%%%%%%%%%%%%%%%%%%%%%%%%%%%%%%%%%%%%%%%%%%%%%%%%%%%%%%%%%%%%%%%%%%%%%%%
\subsection{Space discretization}
We consider a quadrilateral or hexahedral mesh $\triangulation$ and a
corresponding nodal, scalar-valued discontinuous finite element space
$\mathbb{V}_h$ for each component of the hyperbolic system:
\begin{align*}
  \mathbb{V}_h &= \big \{ v_h(\xcoord) \in
  \LtwoScal \;\big|\;
  (v_h \circ \locglobmap_\element)(\widehat{\xcoord}) \in
  \mathbb{Q}^k(\widehat{\element}) \;\forall \element \in \triangulation
    \big\}.
\end{align*}
Here, $\locglobmap_\element:\widehat{\element}\to\element$ denotes a
diffeomorphism mapping the unit square or unit cube $\widehat{\element}$ to
the physical element $\element \in \triangulation$ and $\mathbb{Q}^k(\hat
K)$ is the space of bi-, or trilinear Lagrange polynomials of degree $k$
defined on the reference element $\hat K$. That is, the Lagrangian shape
functions are defined by enforcing the property
$\widehat{\phi}_k(\widehat{\xcoord}_j) = \delta_{jk}$ with the
Gau\ss-Lobatto points $\{\widehat{\xcoord}_k\}_{k \in \mathcal{N}}$ defined
on the reference element. Here, $\mathcal{N}$ is the number of local
degrees of freedom on the cell $K$. The basis functions on the physical
element $K$ are then generated using the reference-to-physical map
$\locglobmap_\element$: More precisely, for each physical element
$\element$, we define shape functions by setting
$\phi_{\element,i}(\xcoord) :=
\widehat{\phi}_i(\locglobmap_\element^{-1}(\xcoord))$ for all $i \in
\mathcal{N}$. More detail on the construction and implementation of finite
element spaces we refer the reader to \cite{Ciarlet1978, GuerErn2044}.

\begin{remark}[Choice of basis functions]
  We have chosen quadrilateral and hexahedral elements for implementational
  convenience. Our framework can accommodate the usual set of simplicial
  elements, tensor-product elements, or even more exotic spatial
  discretizations, such as rational barycentric coordinates on arbitrary
  polygons. The only restriction is that the chosen basis
  $\{\phi_i(\boldsymbol{x})\}$ is interpolatory, has non-negative mass,
  viz., $\int\phi_i(\boldsymbol{x})\text{d}\boldsymbol{x} \ge 0$, and
  satisfies the partition of unity property $\sum_{i \in \mathcal{V}}
  \phi_i(\boldsymbol{x}) = 1$ for all $\boldsymbol{x}$ in the domain. The
  precise location of the nodes is of no consequence.
\end{remark}

We define $\vertices=\big\{i \in
\mathbb{N} \, | \, 1 \leq i \leq \text{dim}(\FESpaceNodalScal)\big\}$
as the index set of global, scalar-valued degrees of freedom corresponding to
$\FESpaceNodalScal$. Similarly, we introduce the set of global shape functions
$\{\phi_{i}(\xcoord)\}_{i \in \vertices}$ and the set of collocation points
$\{\xcoord_{i}\}_{i \in \vertices}$. Note that, in the context of nodal
discontinuous finite elements, different degrees of freedom can be collocated at
the same spatial coordinates. In other words, the situation $\xcoord_{i} =
\xcoord_j$ with $i \not= j \in \vertices$ may occur whenever $\xcoord_{i}$
lies on a vertex, edge or face of the mesh. We introduce the index sets:
\begin{align*}
  \mathcal{I}(\element) &= \big \{ j \in \vertices \;\big|\;
  \text{supp}(\HypBasisScal_j) \cap \element \not = \emptyset \big\},
  \\
  \mathcal{I}(\partial\element) &= \big \{ j \in \vertices \;\big|\;
  \HypBasisScal_j|_{\partial\element} \not \equiv \bzero \big\},
  \\
  \mathcal{I}(\bdry) &= \big \{ j \in \vertices \;\big|\;
  \HypBasisScal_j|_{\bdry} \not \equiv \bzero \big\}.
\end{align*}
Note that the set $\mathcal{I}(\partial\element)$ also contains indices of
shape functions that have no support on $\element$ but on a neighboring element
of $\element$. We also note that when using finite elements with Gau\ss-Lobatto
points the situation $j \in \mathcal{I}(\bdry)$ can only occur if
$\xcoord_j$ lies on the boundary $\bdry$. We assume that the basis
functions satisfy the following partition of unity property for each
element $K$:
\begin{gather}
  \label{partUnit}
  \begin{gathered}
    \sum_{j \in \mathcal{I}(\element)} \phi_{j}(\xcoord) = 1,
    \quad \xcoord \in \element.
  \end{gathered}
\end{gather}

Finally, we introduce some matrices to be used for the
algebraic discretization. We define the consistent mass matrix with entries $m_{ij} \in
\mathbb{R}$ and lumped mass matrix with entries $m_i \in \mathbb{R}$ as
\begin{align}
  \label{eq:mass_matrices}
  m_{ij} \;:=\; \int_{\element} \HypBasisScal_i \HypBasisScal_j\dx,
  \qquad
  m_i \;:=\; \int_{\element} \HypBasisScal_i \dx.
\end{align}
In order to discretize the divergence of the flux, we introduce
a vector-valued matrix
\begin{align}
  \label{eq:cij_def1}
  \bv{c}_{ij} \;:=\;
  \begin{cases}
    \bv{c}_{ij}^{\element_i}-\bv{c}_{ij}^{\partial\element_i}
    &\text{if }j \in \mathcal{I}(\element_i),
    \\
    \bv{c}_{ij}^{\partial\element_i}
    &\text{if }j \in \mathcal{V} \backslash \mathcal{I}(\element_i),
  \end{cases}
\end{align}
where $\element_i \in \triangulation$ is the uniquely defined element satisfying
$\text{supp}(\HypBasisScal_i)\cap K_i \not = \emptyset$, and
\begin{align}
  \label{eq:cij_def2}
  \bv{c}_{ij}^{\element} := \int_{\element} \nabla\HypBasisScal_j
  \HypBasisScal_i \dx
  \, ,  \
  \bv{c}_{ij}^{\partial\element} :=
  \tfrac{1}{2} \int_{\partial K}
  \HypBasisScal_{j} \HypBasisScal_i \normal_\element \ds
  \, , \
  \bv{c}_{i}^{\bdry} :=
  \tfrac{1}{2} \int_{\partial K \cap \bdry}
  \HypBasisScal_i \normal_\element \ds \, ,
\end{align}
where $\normal_\element$ is the outwards pointing normal of the element
$\element$.
The stencil at the node $i$ is defined as
$
  \mathcal{I}(i) = \{j \in \vertices \;|\; \bv{c}_{ij} \not = 0 \}.
$
From definitions \eqref{eq:cij_def1} and \eqref{eq:cij_def2} it follows
that $\bv{c}_{ii} = \bzero$, see Lemma~\ref{PropSkew} for more
details.

%%%%%%%%%%%%%%%%%%%%%%%%%%%%%%%%%%%%%%%%%%%%%%%%%%%%%%%%%%%%%%%%%%%%%%%%%%%%%%%%
%%%%%%%%%%%%%%%%%%%%%%%%%%%%%%%%%%%%%%%%%%%%%%%%%%%%%%%%%%%%%%%%%%%%%%%%%%%%%%%%
%%%%%%%%%%%%%%%%%%%%%%%%%%%%%%%%%%%%%%%%%%%%%%%%%%%%%%%%%%%%%%%%%%%%%%%%%%%%%%%%
%%%%%%%%%%%%%%%%%%%%%%%%%%%%%%%%%%%%%%%%%%%%%%%%%%%%%%%%%%%%%%%%%%%%%%%%%%%%%%%%

\section{Low-order method}
\label{sec:LowOrder}
We now introduce a first-order, graph-based method for
approximating~\eqref{AbstractPDE}. The scheme is based on results
previously reported in~\cite{GuerPop2016, DiscIndep2019, Maier2023}. A
particular novelty of the proposed scheme is the inclusion of boundary
conditions in the formulation, which will be discussed in more detail in
Section~\ref{sec:boundary_conditions}. We introduce the scheme in
Section~\ref{sec:low_order_algebraic_description}, discuss conservation in
Section~\ref{sec:low_order_conservation}, and prove invariant-set and
entropy-dissipation properties in Section~\ref{sec:low_order_stability}.

%%%%%%%%%%%%%%%%%%%%%%%%%%%%%%%%%%%%%%%%%%%%%%%%%%%%%%%%%%%%%%%%%%%%%%%%%%%%%%%%
\subsection{Description of the low-order scheme}
\label{sec:low_order_algebraic_description}
Let
$\uvect_h^n(\xcoord) = \sum_{i\in\vertices}
\statevector_i^n\HypBasisScal_i(\xcoord)$ be a given (discontinuous) finite
element approximation at time $t_n$, where the coefficients shall be
admissible states $\statevector_i^n\in\mathcal{A}$. In addition, let
$\statevector_i^{\bdry,n} \in \mathcal{A}$ with $i \in \mathcal{I}(\bdry)$, be
a given vector of admissible boundary data for time $t_n$. We then
compute the low order update $\statevector_i^{\low,n+1}$ at time
$t_{n+1}=t_n+\dt_n$ as follows:
\begin{multline}
  \label{cijdijscheme}
  m_i \frac{\statevector_i^{\low, n+1} - \statevector_i^n}{\dt_n}
  + \sum_{j \in \mathcal{I}(i)}\big\{\flux(\statevector_j^n)\,\bv{c}_{ij} -
  d_{ij}^{\low,n} (\statevector_j^n - \statevector_i^n)\big\}
  \\
  + \flux(\statevector_i^{\bdry,n})\,\bv{c}_{i}^{\bdry} -
  \loworderdi
  (\statevector_i^{\bdry,n} - \statevector_i^n) = \bzero,
  \quad\text{for }i\in\vertices.
\end{multline}
Scheme \eqref{cijdijscheme} is an algebraic formulation of a discontinuous
Galerkin formulation \cite{DiscIndep2019} in which the underlying weak
formulation is hidden in the matrices $m_i$ and $\bv{c}_{ij}$. We note
that \eqref{cijdijscheme} is based on a \emph{central flux}
approximation~\cite{DiscIndep2019} with subsequent interpolation, using
the nodal property of the shape functions~\cite{Hesthaven2008}:
\begin{multline*}
  \sum_{j \in \mathcal{I}(i)}\flux(\statevector_j^n)\,\bv{c}_{ij}
  \;=\;
  \int_{\element_i}\nabla\cdot\Big(\sum_{j \in
  \mathcal{I}(\element_i)}\flux(\statevector_j^n)
  \HypBasisScal_j\Big)\,\HypBasisScal_i\dx
  \\
  +\frac12\int_{\partial \element_i}\Big(
    \sum_{j \in \mathcal{I}(\partial\element_i)\setminus\mathcal{I}(\element_i)}
    \flux(\statevector_j^n)\,\HypBasisScal_j
    \;-\;
    \sum_{j \in \mathcal{I}(\partial\element_i)\cap\mathcal{I}(\element_i)}
    \flux(\statevector_j^n)\,\HypBasisScal_j
  \Big)\cdot\normal_{\element_i}\, \HypBasisScal_i\dx
  \\
  \;\approx\;
  - \int_{\element_i}\flux(\uvect^n_h)\cdot\nabla\HypBasisScal_i\dx
  +\int_{\partial \element_i}
  \big\{\!\!\big\{\flux(\uvect^n_h)\big\}\!\!\big\}_{\partial\element_i}
  \cdot\normal_{\element_i}\, \HypBasisScal_i\dx,
\end{multline*}
where $\{\!\!\{\flux(\uvect^n_h)\}\!\!\}_{\partial\element_i}$ denotes
the average between the two adjacent states. For more details, see also
\cite{DiscIndep2019}.

Scheme \eqref{cijdijscheme} imposes the boundary data weakly with a jump
condition, defining a boundary flux through an outer state with the help of
boundary conditions in the usual DG form, see e.\,g.,~\cite{ Bassi1997,
Hartmann2002, Hesthaven2008}, albeit with a simplified, \emph{diagonal}
flux, $\flux(\statevector_i^{\bdry,n})\,\bv{c}_{i}^{\bdry}$, as detailed in
Sec.~\ref{sec:boundary_conditions} below. Since $\bv{c}_i=\bzero$ and
$\loworderdi=0$ for $i\in\vertices\setminus\mathcal{I}(\bdry)$,
\eqref{cijdijscheme} is valid for both boundary and interior degrees of
freedom simultaneously. We defer the discussion on how to construct a
suitable boundary data vector $\statevector_i^{\bdry,n}$ to
Section~\ref{sec:boundary_conditions}.

The graph viscosities $d_{ij}^{\low,n}, \loworderdi > 0$ in
\eqref{cijdijscheme} are computed as follows:
\begin{align}
  \label{eq:dij_low_order}
  \begin{cases}
    \begin{aligned}
      d_{ij}^{\low,n}
      &:= |\bv{c}_{ij}|_{\ell^2} \lambda_{\text{max}}^+(\statevector_i^{n},
      \statevector_j^{n}, \normal_{ij}),
      &\text{where }&
      \normal_{ij}= \frac{\bv{c}_{ij}}{|\bv{c}_{ij}|_{\ell^2}},
      \quad\text{for }i\not=j,
      \\[0.25em]
      \loworderdi
      &:=
      |\bv{c}_{i}^{\bdry}|_{\ell^2}
      \lambda_{\text{max}}^+(\statevector_i^{n}, \statevector_i^{\bdry,n},
      \normal_{i}),
      &\text{where }&
      \normal_{i}=
      \frac{\bv{c}_{i}^{\bdry}}{|\bv{c}_{i}^{\bdry}|_{\ell^2}}.
    \end{aligned}
  \end{cases}
\end{align}
Here, $\lambda_{\text{max}}^+(\statevector_L, \statevector_R, \normal):
\mathbb{R}^m \times \mathbb{R}^m \times \mathbb{S}^{d-1} \longrightarrow
\mathbb{R}^+$ is an upper bound of the maximum wavespeed of the
projected Riemann problem \eqref{eq:riemann_problem}~\cite{GuerPop2016}.
We also introduce $d_{ii}^{\low,n}$:
\begin{align}
  \label{eq:dij_low_order_supplement}
  d_{ii}^{\low,n} :=
  -\sum_{j \in \mathcal{I}(i),j\not=i}
  d_{ij}^{\low,n} \;-\; \loworderdi .
\end{align}
This definition plays a role in the computation of the largest admissible time-step
size, see \eqref{CFLcond}, however, we note that
$d_{ii}^{\low,n}$ is not needed in order to compute the update
$\statevector_{i}^{\low,n+1}$ with \eqref{cijdijscheme}.

The low order scheme \eqref{cijdijscheme}--\eqref{eq:dij_low_order} is a
first-order invariant-set preserving approximation in the spirit of
\cite{GuerPop2016}. To this end we rigorously establish conservation
properties in Section~\ref{sec:low_order_conservation}
and derive a \emph{bar state} characterization
that in turn implies an invariant-set property and discrete entropy
inequalities (see Lemma~\ref{EntStabLemma}). We have summarized an
adaptation of \eqref{cijdijscheme}--\eqref{eq:dij_low_order} for the
prescription of boundary conditions in the context of \emph{continuous}
finite element discretizations in Appendix~\ref{AppCg}.

%%%%%%%%%%%%%%%%%%%%%%%%%%%%%%%%%%%%%%%%%%%%%%%%%%%%%%%%%%%%%%%%%%%%%%%%%%%%%%%%
\subsection{Conservation properties}
\label{sec:low_order_conservation}
In order to establish conservation properties of the scheme we make use of
some auxiliary results.

\begin{lemma}[Partition of unity property]\label{LemPart}
  It holds that
  \begin{align}
  \label{PartitionBetter}
    \sum_{j \in \mathcal{I}(i)} \bv{c}_{ij} + \bv{c}_{i}^{\bdry} = \bzero,
    \quad\text{for all } i \in \vertices.
  \end{align}
\end{lemma}

\begin{proof}
  Let $i\in\vertices$ be arbitrary and let $\element_i$ denote the cell
  with $i\in\mathcal{I}(\element)$. Using the definitions \eqref{eq:cij_def1}
  and \eqref{eq:cij_def2}, as well as property \eqref{partUnit}, we can
  compute
  \begin{align*}
    \sum_{j \in \mathcal{I}(i)} \bv{c}_{ij} + \bv{c}_{i}^{\bdry}
    &\;=\;
    \sum_{j \in \mathcal{I}(\element_i)}
    \big(\bv{c}_{ij}^{\element_i}
    - \bv{c}_{ij}^{\partial\element_i} \big)
    \,+\, \sum_{j \in \mathcal{I}(i)\backslash\mathcal{I}(\element_i)}
    \bv{c}_{ij}^{\partial\element_i}
    \,+\, \bv{c}_{i}^{\bdry}
    \\
    &\;=\;
    \int_{\element_i}\nabla
    \Big(\sum_{j \in
\mathcal{I}(\element_i)}\HypBasisScal_j\Big)\,\HypBasisScal_i \dx
    -
    \frac12 \int_{\partial\element_i}
    \Big(\sum_{j \in \mathcal{I}(\element_i)}\HypBasisScal_j\Big)\,
    \HypBasisScal_i \normal_{\element_i}\ds
    \\&\qquad
    \,+\,
    \frac12 \int_{\partial\element_i\setminus\partial\Omega}
    \Big(\sum_{j \in
\mathcal{I}(i)\setminus\mathcal{I}(\element_i)}\HypBasisScal_j\Big)\,
    \HypBasisScal_i \normal_{\element_i}\ds
    \,+\,
    \frac12 \int_{\partial\element_i\cap\partial\Omega} \HypBasisScal_i
    \normal_{\element_i}\ds.
    \\
    &\;=\;
    -
    \frac12 \int_{\partial\element_i} \HypBasisScal_i \normal_{\element_i}\ds
    \,+\,
    \frac12 \int_{\partial\element_i\setminus\partial\Omega}
    \HypBasisScal_i \normal_{\element_i}\ds
    \,+\,
    \frac12 \int_{\partial\element_i\cap\partial\Omega} \HypBasisScal_i
    \normal_{\element_i}\ds
    \\[0.25em]
    &\;=\; \bzero.
  \end{align*}
\end{proof}

\begin{lemma}\label{PropSkew} The matrix $\bv{c}_{ij}$ is
  skew-symmetric:
  \begin{align}
    \label{skewsymm}
    \bv{c}_{ij} &= - \bv{c}_{ji}
    \quad\text{for all }i,j\in\vertices.
  \end{align}
\end{lemma}

\begin{proof}
  The statement is an immediate consequence of definition
  \eqref{eq:cij_def1} and the fact that integration by parts shows
  $\bv{c}_{ij}^{K_i} - \bv{c}_{ij}^{\partial K_i} = - \bv{c}_{ji}^{\partial
  K_i} + \bv{c}_{ji}^{K_i} = -(\bv{c}_{ji}^{\partial K_j} -
  \bv{c}_{ji}^{K_j})$ for $K_i=K_j$; or that $\bv{c}_{ij}^{\partial K_i} =
  \bv{c}_{ji}^{\partial K_i} = - \bv{c}_{ji}^{\partial K_j}$ whenever
  $K_i\not=K_j$ due to the fact that
  $\normal_{\element_i}=-\normal_{\element_j}$ on $\partial K_i\cap
  \partial K_j$.
\end{proof}
Conservation equation \eqref{AbstractPDE} implies that the following
flux-balance of the state $\statevector_i$ between time $t^n$ and $t^{n+1}$
holds true:
\begin{align*}
\int_{\Omega} \state(\xcoord, t^{n+1}) \dx
+ \int_{t^n}^{t^{n+1}} \int_{\bdry} \flux(\state(\xcoord, s))
\normal_{\element} \ds \mathrm{d}s = \int_{\Omega} \state(\xcoord,
t^{n}) \dx.
\end{align*}
We now show that the scheme \eqref{cijdijscheme}
satisfies a discrete counterpart of such a flux balance. We start by
deriving an explicit \emph{skew-symmetric} counterpart of scheme
\eqref{cijdijscheme}:

\begin{lemma}[Skew symmetric form]
  \label{LemSkew}
  Using property \eqref{PartitionBetter} scheme \eqref{cijdijscheme} can be
  written equivalently as
  \begin{align}
    \label{skewSymmetric}
    m_i\,
    \big(\statevector_i^{\low,n+1} - \statevector_i^n\big)
    \;+\;
    \sum_{j \in \mathcal{I}(i)} \bv{F}_{ij}^{\low} +
    \bv{F}_{i}^{\bdry,\low} = \bzero,
  \end{align}
  with the fluxes
  \begin{align*}
    \bv{F}_{ij}^{\low} &:= \dt_n \big(\flux(\statevector_j^n) +
    \flux(\statevector_i^{n})\big)\,\bv{c}_{ij} - \dt_n d_{ij}^{\low,n}
    (\statevector_j^n - \statevector_i^n),
    \\[0.5em]
    \bv{F}_{i}^{\bdry,\low} &:= \dt_n \big(\flux(\statevector_i^{\bdry,n}) +
    \flux(\statevector_i^{n})\big)\,\bv{c}_{i}^{\bdry} - \dt_n \loworderdi
    (\statevector_i^{\bdry,n} - \statevector_i^n),
  \end{align*}
  where the vectors $\bv{F}_{ij}^{\low} \in \mathbb{R}^m$ are skew
  symmetric, i.\,e., $\bv{F}_{ij}^{\low} = - \bv{F}_{ji}^{\low}$.
  Note that $\bv{F}_{i}^{\bdry,\low}$ is a boundary flux.
\end{lemma}
Summing up \eqref{skewSymmetric} over the index $i$ and using the skew
symmetry of $\bv{F}_{ij}^{\low}$ leads to the following corollary.
\begin{corollary}[Total balance]
  \label{TotBalLemma}
  Scheme \eqref{cijdijscheme} satisfies the balance equation
  \begin{align}
    \label{totalBalance}
    \sum_{i \in \vertices} m_i \frac{\statevector_i^{n+1} -
    \statevector_i^{n}}{\dt_n} + \sum_{i \in \mathcal{I}(\bdry)}
    \bv{F}_{i}^{\bdry,\low} = 0
  \end{align}
  for all $i \in \vertices$. Using definitions \eqref{eq:mass_matrices} and
  \eqref{eq:cij_def2} one can write \eqref{totalBalance} equivalently as
  \begin{multline*}
    \int_\Omega \state_h^{n+1}(\xcoord) \dx
    \;+\;
    \frac{\dt_n}{2}\,\int_{\partial\Omega}
    \sum_{i\in\mathcal{I}(\bdry)}\HypBasisScal_i
    \big(\flux(\statevector_i^{\bdry,n}) + \flux(\statevector_i^{n})\big)
    \,\normal_\element\ds
    \\
    \;-\;
    \sum_{i\in\mathcal{I}(\bdry)} \loworderdi \big(\statevector_i^{\bdry,n}
    - \statevector_i^n\big)
    \;=\;
    \int_\Omega \state_h^{n}(\xcoord) \dx.
  \end{multline*}
\end{corollary}
\begin{remark}
  $\bv{F}_{i}^{\bdry,\low}$ can be viewed as a central flux between the
  boundary state $\statevector_i^{n}$ and a \emph{ghost node} with state
  $\statevector_i^{\bdry,n}$ in the usual discontinuous Galerkin framework.
\end{remark}

%%%%%%%%%%%%%%%%%%%%%%%%%%%%%%%%%%%%%%%%%%%%%%%%%%%%%%%%%%%%%%%%%%%%%%%%%%%%%%%%
\subsection{Invariant-set preservation}
\label{sec:low_order_stability}
We now focus on stability properties of the low-order scheme. In the spirit
of \cite{GuerPop2016, DiscIndep2019}, we rewrite \eqref{cijdijscheme} as a
convex combination of \emph{bar states}. These are states formed by an
algebraic combination of interacting degrees of freedom that, under a
suitable CFL condition, correspond to the spatial average of an associated
one dimensional Riemann problem. However, in contrast to the discussion
found in the references above, we end up with an additional set of bar
states that depend on the boundary data. We start with the following
algebraic identity.

\begin{lemma}[Convex reformulation]
  \label{ConveRef}
  The update procedure \eqref{cijdijscheme} can equivalently be written as
  follows:
  \begin{multline}
    \label{ConvexRef}
    \statevector_i^{\low,n+1} = \left(1 +\frac{2 \dt_n d_{ii}^{\low,n}}{m_i}\right)
    \statevector_{i}^{n}
    + \frac{2 \dt_n \loworderdi}{m_i} \overline{\statevector}_{i}^{\bdry,n}
    + \sum_{j \in \mathcal{I}(i)\backslash\{i\}} \frac{2
    \dt_n d_{ij}^{\low,n}}{m_i} \overline{\statevector}_{ij}^{n},
  \end{multline}
  where $\overline{\statevector}_{ij}^{n}$ and
  $\overline{\statevector}_{i}^{\bdry,n}$
  are the bar states defined by,
  \begin{align}
    \label{UsualBarState}
    \overline{\statevector}_{ij}^{n} &= \frac{1}{2}(\statevector_j^{n} +
    \statevector_i^{n})
    - \frac{|\bv{c}_{ij}|}{2 d_{ij}^\low} \left(\flux(\statevector_j^{n}) -
    \flux(\statevector_i^{n})\right) \normal_{ij}, \\
    \label{BoundaryBarState}
    \overline{\statevector}_{i}^{\bdry,n} &=
    \frac{1}{2}(\statevector_i^{\bdry,n} + \statevector_i^{n})
    - \frac{|\bv{c}_i^{\bdry}|}{2 \loworderdi}
    \left(\flux(\statevector_i^{\bdry,n}) -
    \flux(\statevector_i^{n})\right)
    \normal_{i}.
  \end{align}
Note that $\overline{\statevector}_{i}^{\bdry,n}$ depends on the boundary data
$\statevector_i^{\bdry,n}$.
\end{lemma}

\begin{proof}
  Similarly to the derivation of the skew symmetric form
  \eqref{skewSymmetric}, we can use identity \eqref{PartitionBetter} to
  rewrite \eqref{cijdijscheme} as follows:
  \begin{multline*}
    m_i \statevector_i^{n+1}
    = m_i \statevector_i^n
    - \dt_n \sum_{j \in \mathcal{I}(i)} \big(\flux(\statevector_j^n)-
    \flux(\statevector_i^{n})\big) \bv{c}_{ij} - d_{ij}^{\low,n}
    (\statevector_j^n - \statevector_i^n) \\
    - \dt_n \big(\flux(\statevector_i^{\bdry,n}) -
    \flux(\statevector_i^{n})\big)\bv{c}_{i}^{\bdry} - \loworderdi
    (\statevector_i^{\bdry,n} - \statevector_i^n).
  \end{multline*}
  We now add and subtract $2 \dt_n \sum_{j \in
  \mathcal{I}(i)\backslash \{i\}} d_{ij}^{\low,n} \statevector_i^n$ and $2
  \dt_n \loworderdi \statevector_i^{\bdry,n}$:
  \begin{multline*}
    m_i \statevector_i^{n+1}
    \;=\; \big(m_i - 2 \dt_n \loworderdi - \sum_{j \in
    \mathcal{I}(i)\backslash \{i\}} 2 \dt_n d_{ij}^{\low,n}\big)
    \,\statevector_i^n
    \\
    - \dt_n \sum_{j \in \mathcal{I}(i)\backslash \{i\}}
    \big(\flux(\statevector_j^n)- \flux(\statevector_i^{n})\big)
    \bv{c}_{ij} - d_{ij}^{\low,n} (\statevector_j^n + \statevector_i^n)
    \\
    - \dt_n \big(\flux(\statevector_i^{\bdry,n}) -
    \flux(\statevector_i^{n})\big)\bv{c}_{i}^{\bdry} - \loworderdi
    (\statevector_i^{\bdry,n} + \statevector_i^n).
  \end{multline*}
  The result now follows readily dividing both sides of the equality by $m_i$
  and using definition \eqref{eq:dij_low_order_supplement}.
\end{proof}

\begin{lemma}[Pointwise entropy inequality]
  \label{EntStabLemma}
  Let $\{\eta, \eflux\}$ be \emph{any} entropy-flux pair of the hyperbolic
  system $\partial_t \state + \diver{}\flux(\state) =
  \bzero$~\cite{Viscous2014, GuerPop2016}. Assume that the update
  \eqref{cijdijscheme} is performed with a time step size $\dt_n$
  satisfying the following CFL condition:
  \begin{align}
    \label{CFLcond}
    - \dt_n \, \frac{2\,d_{ii}^{\low,n}}{m_i} \leq 1
    \quad \text{for all } i \in \vertices.
  \end{align}
  Then the update $\statevector_i^{n+1}$ satisfies the following
  pointwise entropy inequality:
  \begin{multline}
    \label{EntropyIneq}
    m_i \frac{\eta(\statevector_i^{n+1}) - \eta(\statevector_i^{n})}{\dt_n}
    + \sum_{j \in \mathcal{I}(i)} \eflux(\statevector_j^{n}) \cdot \bv{c}_{ij}
    - d_{ij}^{\low,n} (\eta(\statevector_j^{n}) - \eta(\statevector_i^{n}))
    \\
    + \eflux(\statevector_i^{\bdry,n}) \cdot \bv{c}_i^{\bdry} - \loworderdi
    (\eta(\statevector_i^{\bdry,n}) - \eta(\statevector_i^{n})) \leq 0
  \end{multline}
  for all $i \in \vertices$.
\end{lemma}

\begin{proof}
  Proving inequality \eqref{EntropyIneq} relies on the convex reformulation
  \eqref{ConvexRef}, the convexity of the entropy $\eta$ and corresponding
  inequalities that hold true for the bar states defined in
  \eqref{UsualBarState} and \eqref{BoundaryBarState}:
  \begin{align*}
    \eta(\overline{\statevector}_{ij}^{n}) &\leq
    \frac{1}{2}\big(\eta(\statevector_i^n) + \eta(\statevector_j^n)\big)
    - \frac{|\bv{c}_{ij}|_{\ell^2}}{2 d_{ij}^\low}
    \big(\eflux(\statevector_j^n) - \eflux(\statevector_i^n)\big)\normal_{ij},
    \\
    \eta(\overline{\statevector}_{i}^{\bdry,n}) &\leq \frac{1}{2}
    \big(\eta(\statevector_i^{\bdry,n}) + \eta(\statevector_i^n)\big) -
    \frac{|\bv{c}^{\bdry}_{i}|_{\ell^2}}{2 d_{i}^{\bdry,n}}
    \big(\eflux(\statevector_i^{\bdry,n}) -
    \eflux(\statevector_i^n)\big)\normal_{i}.
  \end{align*}
  Such inequalities hold true provided that $d_{ij}^{\low,n}$ and
  $d_{i}^{\bdry,n}$ are chosen large enough so that the bar states represent
  an \emph{average value} over the \emph{Riemann fan}
  \cite{GuerPop2016}. In particular, the choice \eqref{eq:dij_low_order} is
  sufficient \cite{GuerPop2016}. For further details we refer to the
  detailed discussion found in \cite[Thm.\,4.7]{GuerPop2016} and
  \cite[Thm.\,3.8]{DiscIndep2019}.
\end{proof}

\begin{remark}[Global entropy inequality]
  Similarly to the procedure in Section~\ref{sec:low_order_conservation}
  that establishes global conservation (see Lemma~\ref{LemSkew} and
  Corollary~\ref{TotBalLemma}), we can rewrite inequality
  \eqref{EntropyIneq} in skew symmetric form and sum up. This leads to a
  global entropy inequality:
  \begin{align}
    \label{EntropyIneqGlobal}
    \sum_{i \in \vertices} m_i \eta(\statevector_i^{n+1}) + \dt_n
    \sum_{i \in \mathcal{I}(\bdry)} \bv{Q}_{i}^{\bdry,\low}
    \;\le\;
    \sum_{i \in \vertices} m_i \eta(\statevector_i^{n}),
  \end{align}
  with (viscous) boundary fluxes
  \begin{align*}
    \bv{Q}_{i}^{\bdry,\low} :=
    \left(\eflux(\statevector_i^{\bdry,n}) +
    \eflux(\statevector_i^{n})\right)\cdot\bv{c}_i^{\bdry} - \loworderdi
    \left(\eta(\statevector_i^{\bdry,n}) - \eta(\statevector_i^{n})\right).
  \end{align*}
  This is nothing else than the discrete counterpart of a global entropy
  inequality satisfied by the entropy-flux pair $\{\eta, \eflux\}$.
\end{remark}

\begin{remark}[Time step size restriction]
  CFL condition \eqref{CFLcond} determines the largest time step size that
  can be used for an individual update step. For example, in practical
  implementations it is convenient to select a (user specified) constant
  $0<\text{Cr}\leq1$ and then compute a time step size as follows:
  \begin{align}\label{MaxCFLdt}
    \dt_n = \,\text{Cr} \cdot \min_{i \in \vertices}
    \Big(- \frac{m_i}{2\,d_{ii}^{\low,n}}\Big).
  \end{align}
  Here, $m_i$ decreases and $d_{ii}^{L,n}$ grows with increasing polynomial
  degree $k$ leading to a stricter time step size restriction.
\end{remark}

\begin{lemma}[Invariant set property]
  \label{InvariantProp}
  Under the stated CFL condition \eqref{CFLcond} and assuming that the
  provided boundary data $\statevector_i^{\bdry,n}$ is in the invariant set
  $\mathcal{A}$ for all $i \in \mathcal{I}(\bdry)$, then the update
  $\statevector_i^{n+1}$ computed by \eqref{cijdijscheme} and
  \eqref{eq:dij_low_order} will satisfy $\statevector_i^{n+1} \in
  \mathcal{A}$ as well.
\end{lemma}

\begin{proof}
  The statement is a direct consequence of the fact that \eqref{ConvexRef}
  expresses $\statevector_i^{n+1}$ as a convex combination of bar states,
  that in turn are located in the invariant set $\mathcal{A}$ provided that
  the CFL condition \eqref{CFLcond} holds; see Lemma~\ref{EntStabLemma}.
\end{proof}

%%%%%%%%%%%%%%%%%%%%%%%%%%%%%%%%%%%%%%%%%%%%%%%%%%%%%%%%%%%%%%%%%%%%%%%%%%%%%%%%
%%%%%%%%%%%%%%%%%%%%%%%%%%%%%%%%%%%%%%%%%%%%%%%%%%%%%%%%%%%%%%%%%%%%%%%%%%%%%%%%
%%%%%%%%%%%%%%%%%%%%%%%%%%%%%%%%%%%%%%%%%%%%%%%%%%%%%%%%%%%%%%%%%%%%%%%%%%%%%%%%
\section{Boundary conditions}
\label{sec:boundary_conditions}
We now discuss how to construct the boundary data vector
$\statevector_i^{\bdry,n} \in \mathbb{R}^m$, $i \in \mathcal{I}(\bdry)$ for
different types of boundary conditions.

The construction of the boundary data vector $\statevector_i^{\bdry,n}$
follows well established procedures, we refer to e.\,g.,~\cite{
Hedstrom_1979, Demkowicz_etal_1990, Bassi1997, Hartmann2002, Hesthaven2008}. For the
sake of completeness we briefly summarize our approach based on
\cite{Guermond2022}.

%%%%%%%%%%%%%%%%%%%%%%%%%%%%%%%%%%%%%%%%%%%%%%%%%%%%%%%%%%%%%%%%%%%%%%%%%%%%%%%%
\subsection{Construction of boundary data $\statevector_i^{\bdry,n}$}

We distinguish Dirichlet boundary conditions, slip boundary conditions,
supersonic and subsonic in- and outflow.

\paragraph{Dirichlet boundary conditions}
For Dirichlet boundaries we simply set $\statevector_i^{\bdry,n} \in
\mathcal{A}$ to the desired boundary data at position $\xcoord_i$ for time
$t_n$.

\paragraph{Slip boundary conditions}
We impose slip boundary conditions for a boundary state $\statevector_i^{n}
= [\rho_i^n, \mom_i^{n}, \totme_i^n]^\transp$ at a boundary collocation
point $\xcoord_i$ by setting
\begin{align}
  \label{eq:slip_boundary}
  \statevector_i^{\bdry,n} := [\rho_i^n, \mom_i^{\bdry,n},
  \totme_i^n]^\transp,
  \quad\text{where}\quad
  \mom_i^{\bdry,n} := \mom_i^n - 2 (\mom_i^n\cdot\normal_i) \normal_i,
\end{align}
and where we recall the definition $\normal_i = \bv{c}_{i}^{\bdry} /
|\bv{c}_{i}^{\bdry}|_{\ell^2}$ for a boundary collocation point
$\xcoord_i$. This implies that $\mom_i^{\partial,n}$ and $\mom_i^n$ have
opposite normal components but the same tangential projection with respect
to the normal $\normal_i$. The boundary flux $\bv{F}_{i}^{\bdry,\low}$
consequently only affects the balance of the normal component of the
momentum and leaves all other components unaffected.

\paragraph{Supersonic in- and outflow}
In order to impose supersonic in- and outflow at portions of the boundary, we
proceed as follows. Given a state $\statevector_i^{n} = [\rho_i^n,
\mom_i^{n}, \totme_i^n]^\transp$ at a boundary collocation point
$\xcoord_i$ with velocity $\vel_i^n$, local speed of sound
$a_i(\statevector_i^{n})$, and corresponding Dirichlet data
$\statevector_i^{\text{d},n}$ for the inflow, we set
\begin{align}
  \label{eq:supersonic_inoutflow}
  \statevector_i^{\bdry, n}
  \;=\;
  \begin{cases}
    \begin{aligned}
      &\statevector_i^{\text{d},n} &\quad&
      \text{if}\; \vel_i^n\cdot\normal_i < -a_i,
      \\[0.5em]
      &\statevector_i^{n} &\quad&
      \text{if}\; \vel_i^n\cdot\normal_i \ge a_i.
    \end{aligned}
  \end{cases}
\end{align}

\paragraph{Subsonic in-flow and outflow}
For the subsonic case, we need to distinguish in- and outgoing
characteristics. To impose conditions only on the ingoing ones, we
construct the boundary data vector $\statevector_i^{\bdry,
n}$ by blending together the current state
$\statevector_i^{n}$ and the given Dirichlet data
$\statevector_i^{\text{d},n}$~\cite{Hedstrom_1979, Demkowicz_etal_1990,
Guermond2022}. To this end, we briefly review the approach described
in \cite{Guermond2022} for the case of our discontinuous formulation.

Given a state $\statevector = [\rho, \mom, \totme]^\transp$ and a unit
vector $\normal$ we introduce the following set of characteristic variables
$\mathcal{R}_k(\statevector, \normal)$ and characteristic speeds
$\lambda_k(\statevector,\normal)$, $k=1,\ldots,4$:
\begin{align}
  \label{eq:characteristics}
  \begin{cases}
    \begin{aligned}
      \mathcal{R}(\statevector,\normal)\;:=\;\Big\{
        &v_{\normal} - \tfrac{2 a}{\gamma - 1}, &\;&
        \frac{p(\statevector)}{\rho^\gamma}, &\;&
        \vel - (\vel\cdot\normal) \normal, &\;&
        v_{\normal} + \tfrac{2 a}{\gamma - 1}
      &\Big\},
      \\[0.25em]
      \lambda(\statevector,\normal)\;:=\;\big\{
        &v_{\normal} - a, &\;&
        v_{\normal}, &\;&
        v_{\normal}, &\;&
        v_{\normal} + a
      &\big\}.
    \end{aligned}
  \end{cases}
\end{align}
Here, $v_{\normal} = \vel \cdot \normal$ and $a =
\sqrt{\gamma\frac{p}{\rho}}$ is the local speed of sound. The strategy now
consists of constructing the boundary data satisfying:
$\statevector_i^{\bdry, n}$ with $\mathcal{R}_i(\statevector_i^{\bdry,
n},\normal) = \mathcal{R}_i(\statevector_i^{\text{d},n},\normal)$ if
$\lambda_i(\statevector_i^{n}, \normal) \le 0$ (incoming characteristics), and
$\mathcal{R}_i(\statevector_i^{\bdry, n},\normal) =
\mathcal{R}_i(\statevector_i^{n},\normal)$ otherwise. Constructing an
admissible state $\statevector_i^{\bdry,n}$ satisfying such constraint is always
possible; see~\cite[\S\,4.3.2]{Guermond2022}. In addition, the resulting
state $\statevector_i^{\bdry, n}$ is always admissible provided
$\statevector_i^{\text{d}}$ is admissible, and the construction coincides
with \eqref{eq:supersonic_inoutflow} for the case of supersonic
in- and outflow.

%%%%%%%%%%%%%%%%%%%%%%%%%%%%%%%%%%%%%%%%%%%%%%%%%%%%%%%%%%%%%%%%%%%%%%%%%%%%%%%%
\subsection{Multi-valued boundary conditions}
\label{rem:MValuedBC}
As opposed to classical implementations of discontinuous Galerkin schemes via
face integrals, the algebraic approach implies
that more than one boundary condition applies to a given
boundary collocation point $\xcoord_i$, for instance:
\begin{itemize}
  \item
    For a typical channel flow setup, a small subset of the boundary
    collocation points $\xcoord_i$ (those at corners) lie between the slip
    boundaries at the top and bottom as well as the inflow and outflow
    boundaries at the left and right, respectively.
  \item
    Another example is given by two parts of the boundary with slip
    boundary conditions that meet at an angle, say $90^\circ$. Here, slip
    boundary conditions for both normals should rather be enforced instead
    of a single slip boundary condition with a combined normal (of
    $45^\circ$).
\end{itemize}
In order to treat such boundary states, we first partition the boundary
$\partial\Omega$ into all $\mathcal{K}$ disjoint components $\partial\Omega_k$,
where we either apply a different boundary condition, or where portions of the
boundary meet with a large angle. Then we split $\bv{c}_i^{\bdry}$
accordingly:
\begin{align*}
  \partial\Omega = \bigcup_{k\in[1:\mathcal{K}]}\partial\Omega_k,
  \quad
  \bv{c}_i^{\bdry} = \sum_{k\in[1:\mathcal{K}]} \bv{c}_i^{\bdry,k},
  \quad\text{with}\quad
  \bv{c}_i^{\bdry,k} := \tfrac{1}{2} \int_{\partial K \cap
  \bdry_{k}} \HypBasisScal_i \normal_\element \ds.
\end{align*}
Finally, scheme \eqref{cijdijscheme} takes the form
\begin{multline}
  \label{cijdijschememodified}
  m_i \frac{\statevector_i^{\low, n+1} - \statevector_i^n}{\dt_n}
  + \sum_{j \in \mathcal{I}(i)}\big\{\flux(\statevector_j^n)\,\bv{c}_{ij} -
  d_{ij}^{\low,n} (\statevector_j^n - \statevector_i^n)\big\}
  \\
  + \sum_{k\in[1:\mathcal{K}]}\Big\{
  \flux(\statevector_i^{\bdry,k,n})\,\bv{c}_{i}^{\bdry,k} -
  d_{i}^{\bdry, k, n} (\statevector_i^{\bdry,k,n} - \statevector_i^n)
  \Big\}
  = \bzero, \quad\text{for }i\in\vertices,
\end{multline}
where $\statevector_i^{\bdry,k,n}$ is a modified vector of appropriately
chosen boundary data, and the modified graph viscosity $d_{i}^{\bdry, k, n}$
is computed as follows:
\begin{align}
  \label{eq:dij_low_order_modified}
  d_{i}^{\bdry, k, n}
  :=
  |\bv{c}_{i}^{\bdry,k}|_{\ell^2}\,
  \lambda_{\text{max}}^+(\statevector_i^{\bdry,k,n}, \statevector_i^{n},
  \normal_{i}^k),
  \quad \text{where } \quad
  \normal_{i}^k=
  \frac{\bv{c}_{i}^{\bdry,k}}{|\bv{c}_{i}^{\bdry,k}|_{\ell^2}}.
\end{align}

%%%%%%%%%%%%%%%%%%%%%%%%%%%%%%%%%%%%%%%%%%%%%%%%%%%%%%%%%%%%%%%%%%%%%%%%%%%%%%%%
%%%%%%%%%%%%%%%%%%%%%%%%%%%%%%%%%%%%%%%%%%%%%%%%%%%%%%%%%%%%%%%%%%%%%%%%%%%%%%%%
%%%%%%%%%%%%%%%%%%%%%%%%%%%%%%%%%%%%%%%%%%%%%%%%%%%%%%%%%%%%%%%%%%%%%%%%%%%%%%%%

\section{High-order method and convex limiting}
\label{sec:highorder}

Following the same approach as discussed in \cite{Euler2018,DiscIndep2019},
we now introduce a formally high-order method by using the consistent mass
matrix and introducing a \emph{high-order graph-viscosity}
$d_{ij}^{\high}$,
\begin{multline}
  \label{HighScheme}
  \sum_{j \in \mathcal{I}(i)}
  m_{ij} \frac{\statevector_j^{\high, n+1} - \statevector_j^n}{\dt_n}
  + \sum_{j \in \mathcal{I}(i)}\big\{\flux(\statevector_j^n)\,\bv{c}_{ij} -
  d_{ij}^{\high,n} (\statevector_j^n - \statevector_i^n)\big\}
  \\
  + \flux(\statevector_i^{\bdry,n})\,\bv{c}_{i}^{\bdry} -
  \highorderbdrydi
  (\statevector_i^{\bdry,n} - \statevector_i^n) = \bzero,
  \quad\text{for }i\in\vertices.
\end{multline}
Here, $m_{ij} = \int_{\element} \HypBasisScal_i \HypBasisScal_j \dx$
denotes the consistent mass matrix. A considerable body of stabilization
techniques have been developed over the years, supplying ideas that could be
adapted to the computation of high-order graph viscosities $d_{ij}^{\high,n}$.
Among these methods we mention entropy-viscosity \cite{Guer2011}, smoothness
sensors \cite{Persson2006}, and semi-discrete entropy-stable flux constructions
\cite{Fjord2012, Fisher2013}; see also \cite{Michoski2016} for a comprehensive
review of approaches. All of these methods have in common that they try to
ensure that $d_{ij}^{\high,n} \approx d_{ij}^{\low,n}$ near shocks and
discontinuities, while forcing $d_{ij}^{\high,n} \approx 0$ in smooth regions of
the solution.

The development of high-order methods for discontinuous spatial
discretizations requires some attention to the minimal amount of viscosity.
Without enough viscosity between the element interfaces, the method might not
even be stable for smooth solutions. Therefore, we will
first present a high-order viscosity $d_{ij}^{\text{min},n}$ such that if
used in scheme \eqref{HighScheme}: (i) it results in a stable approximation
of smooth solutions on structured and unstructured meshes; and (ii) we
observe optimal convergence rates for smooth problems. We will call such a
viscosity $d_{ij}^{\text{min},n}$ a \emph{minimally stabilizing
viscosity}. This viscosity will not have any shock-capturing capability and
might not be best choice of viscosity for non-smooth problems. The sole
purpose of such a viscosity is to define a minimal amount of viscosity that
the method should always have.

On the other hand, we would like to adapt the entropy viscosity methodology
described in \cite{Guer2011, Guermond2014} to the context of graph-based
methods using discontinuous spatial discretizations. We will denote the
entropy viscosity as $d_{ij}^{\text{ev},n}$. Such viscosity should provide
the required shock capturing capabilities. Ultimately, we set
\begin{align}\label{HighOrderViscosity}
  d_{ij}^{\high,n} := \max \{
  d_{ij}^{\text{min},n},
  d_{ij}^{\text{ev},n} \}
\end{align}
to guarantee that the high-order scheme \eqref{HighScheme} possesses
enough viscosity to deliver stable solutions and optimal convergence rates in
the context of smooth solutions as well as shock-capturing capabilities in the
context of non-smooth problems.

In Section \ref{SubMinVisc} we define the minimally stabilizing viscosity
$d_{ij}^{\text{min}}$ while in Section \ref{SubMinEV} we describe the
entropy viscosity $d_{ij}^{\text{ev},n}$. In Section \ref{SubAlgFlux} we
describe the convex limiting procedure to ensure that the blended method
satisfies \emph{local} bounds at every collocation point.

%%%%%%%%%%%%%%%%%%%%%%%%%%%%%%%%%%%%%%%%%%%%%%%%%%%%%%%%%%%%%%%%%%%%%%%%%%%%%%%%
\subsection{Minimally stabilizing high-order viscosity}
\label{SubMinVisc}

The usual Lax--Friedrichs flux of the form $\int_{F} \lambda [\![ u_h
]\!]\phi_i \ds$, where $\lambda$ is an estimate on the maximum wavespeed between
cell interfaces, leads to the usual optimal convergence rate $\|u -
u_h\|_{L^2(\Omega)} \leq \mathcal{O}(h^{k+\frac{1}{2}})$. However, it can
be observed experimentally that the behaviour of the Lax--Friedrichs flux is
suboptimal for the case of even polynomial degree in the $L^1$-norm. For
instance, for $k = 2$ when solving the isentropic vortex \cite{Zhang2010,
Hest2020, Chan2023} a rate of $\mathcal{O}(h^{2.75})$ can be observed
instead of the expected $\mathcal{O}(h^{3.0})$, see Remark
\ref{rem:asymptotic_regime}. The phenomenon of observing suboptimal
convergence rates for even polynomial degrees in the $L^1$ norm has been
well known for a while by practitioners but it is rarely commented on in
the literature. Following an argument in \cite{Zheng2023} we use a
different scaling for the cases of odd and even polynomial degrees in order
to define a minimally stabilizing high-order viscosity that recovers the
optimal rate $\mathcal{O}(h^{k+1})$ in the $L^1$-norm when $k$ is even.

\begin{remark}
\label{rem:asymptotic_regime}
We note that such a degradation in convergence rates typically manifests
only after a sufficiently large number of mesh refinements has been reached.
Therefore, in our numerical results (reported in Section~\ref{CompExperiments})
a sufficiently large number of mesh refinements is used to ensure that we have reached
an \emph{asymptotic regime}.
\end{remark}

In light of the discussion above, we set $d_{ij}^{\text{min}} :=
d_{ij}^{\low,n}$ if $\xcoord_i = \xcoord_j$, for the case of \emph{odd}
polynomial degree. This choice is roughly equivalent the
interfacial Lax--Friedrichs flux. However, for \emph{even} polynomial degree we
adapt the idea outlined in \cite{Zheng2023} and set $d_{ij}^{\text{min}} =
\mathcal{O}(h^\frac{1}{2}) d_{ij}^{\low,n}$ if $\xcoord_i = \xcoord_j$.
However, we want to avoid introducing a length scaling $h$ into the high-order
viscosity, therefore we set
\begin{align}
  \label{ViscMin}
  d_{ij}^{\text{min}} :=
  \begin{cases}
    c_k \widehat{h}_{ij}^{p_k} d_{ij}^{\low,n} &\text{if } \xcoord_i =
    \xcoord_j,
    \\[0.25em]
    0 &\text{otherwise},
  \end{cases}
\end{align}
where $\widehat{h}_{ij} := \big(\tfrac{1}{2} \tfrac{m_i + m_j}{|\Omega|}
\big)^{\frac{1}{d}}$ is a dimensionless mesh size, and the constants $c_k$ and
$p_k$ are set to $c_k=1$, and
\begin{align*}
  p_k = \begin{cases}
    \frac{1}{2} \text{ if } k \text{ is even},
    \\
    0 \text{ if } k \text{ is odd}.
  \end{cases}\quad
\end{align*}
The graph viscosity $d_{ij}^{\text{min}}$ is symmetric by construction.
The choice of coefficient $p_k$ follows from the theoretical and
computational discussion outlined in \cite{Zheng2023}. Extensive
numerical tests indicate that the power $p = \frac{1}{2}$ for the case of
even polynomial degree is indeed optimal; it maintains a stable
approximation of smooth solutions on unstructured meshes, as well as
optimal convergence rates in the $L^1$ norm. On the other hand, the
non-dimensional constant $c_k>0$ can be chosen more freely. For the sake of
reproducibility we simply report our particular choice for the parameters used
in our numerical results (Section~\ref{CompExperiments}), which we found to be
reasonable for a large number of test cases.

\begin{remark}[Superconvergence]
  We note that discontinuous spatial discretizations of even polynomial
  degree (without stabilization) are superconvergent for the case of
  linear conservation equations on uniform meshes. For instance,
  non-stabilized discontinuous $\mathbb{Q}^2$ spatial discretizations have been
  shown to be fourth order accurate for smooth linear problems on uniform
  meshes with periodic boundary conditions, see\cite{Ains2014}. In view of
  these theoretical results, it seems tempting to simply set
  $d_{ij}^\text{min} = 0$ as a minimal viscosity choice. However, we have
  observed numerically that these theoretical results for linear
  conservation equations do not necessarily translate to the case of
  non-linear hyperbolic systems, nor to general unstructured hexahedral
  meshes.
\end{remark}

%%%%%%%%%%%%%%%%%%%%%%%%%%%%%%%%%%%%%%%%%%%%%%%%%%%%%%%%%%%%%%%%%%%%%%%%%%%%%%%%
\subsection{Entropy viscosity}
\label{SubMinEV}

The entropy viscosity commutator has been introduced in \cite{Guer2011}.
Here, we summarize a variant discussed in \cite{Guermond2014}. Consider the
generalized Harten entropy $f\big(s(\state)\big)$, where $s(\state)$ is
the specific entropy~\eqref{EntropyEuler}, and $f$ is any function
satisfying the constraints
\begin{align*}
  f'(s) > 0,
  \; f'(s) c_p^{-1} - f''(s) > 0 .
\end{align*}
Here $c_p = \theta \frac{\partial s}{\partial \theta}$ is the specific heat at
constant pressure, see \cite{Viscous2014}. For any admissible state $\state =
[\rho, \mom, \totme]^\transp \in \mathbb{R}^{d+2}$ we adopt the shorthand
notation $f(\state) := f\big(s(\state)\big)$. We then
define a \emph{shifted} generalized mathematical entropy $\Phi_i^n(\state)$ and
a corresponding entropy-flux $\eflux_i^n(\state)$:
\begin{align*}
  \Phi_i^n(\state) = \rho \big[f(\state) - f(\statevector_i^n)\big],
  \qquad
  \eflux_i^n(\state) = \mom \big[f(\state) - f(\state_i^n)\big].
\end{align*}
Let $\nabla_{\state}\Phi_i^n(\state) \in \mathbb{R}^{d+2}$ denote the gradient
of $\Phi_i^n(\state)$ with respect to the state $\state$ and set:
\begin{align*}
  R_i &:= \sum_{j \in \mathcal{I}(i)} \big[\eflux_i(\statevector_j^n) -
  (\nabla_{\state}\Phi_i^n)^\transp \flux(\statevector_j^n)\big] \cdot
  \bv{c}_{ij}
  \\
  D_i &:= \big| \sum_{j \in \mathcal{I}(i)}
  \eflux(\statevector_j^n)\cdot\bv{c}_{ij} \big|
  + \sum_{k \in m} \big|[\nabla_{\state}\Phi_i]_k\big|
  \big|\flux_k(\statevector_j^n)\bv{c}_{ij}\big|
\end{align*}
where $[\nabla_{\state}\Phi_i]_k$ is the $k$-th component of
$\nabla_{\state}\Phi_i$ and $\flux_k(\statevector_j^n) \in \mathbb{R}^{1 \times
d}$ denotes the $k$-th row of the flux $\flux(\statevector_j^n) \in
\mathbb{R}^{(d+2)\times d}$. We define the normalized entropy-viscosity
residual $N_i$ and entropy viscosity $d_{ij}^{\text{ev}}$ as
\begin{align*}
  N_i := \frac{R_i}{D_i}
  \quad\text{and}\quad
  d_{ij}^{\text{ev},n} := d_{ij}^\low \min \big\{ c_\text{ev}
  \max\{|N_i|,|N_j|\big\}, 1\},
\end{align*}
where $c_{\text{ev}}$ is a constant that will in general depend on the
polynomial degree. From numerical explorations we have chosen to use
$c_{\text{ev}} = 1, 0.5, 0.25$ for the polynomial degrees $k = 1,2,3$,
respectively.

%%%%%%%%%%%%%%%%%%%%%%%%%%%%%%%%%%%%%%%%%%%%%%%%%%%%%%%%%%%%%%%%%%%%%%%%%%%%%%%%
\subsection{Convex limiting: algebraic reformulation}
\label{SubAlgFlux}
In analogy to Lemma~\ref{LemSkew}, we rewrite the high-order scheme
\eqref{HighScheme} as follows:
\begin{align}
  \label{HighSchemeAlgebraic}
  m_i (\statevector_i^{\high,n+1} - \statevector_i^n) + \sum_{j \in
  \mathcal{I}(i)} \bv{F}_{ij}^{\high} + \bv{F}_{i}^{\bdry,\high} = \bzero,
\end{align}
where the algebraic fluxes $\bv{F}_{ij}^{\high}$ are given by
\begin{align*}
  \bv{F}_{ij}^{\high} &:= \dt_n \big(\flux(\statevector_j^n) +
  \flux(\statevector_i^{n})\big)\,\bv{c}_{ij} - \dt_n d_{ij}^{\high,n}
  (\statevector_j^n - \statevector_i^n)
  \\
  &\qquad\qquad\qquad\qquad\qquad\qquad
  + (m_{ij} - \delta_{ij} m_i) (\statevector_j^{\high,n+1} - \statevector_j^n
  -\statevector_i^{\high, n+1} + \statevector_i^n),
  \\[0.5em]
  \bv{F}_{i}^{\bdry,\high} &:= \dt_n \big(\flux(\statevector_i^{\bdry,n}) +
  \flux(\statevector_i^{n})\big)\,\bv{c}_{i}^{\bdry} - \dt_n
  d_{ij}^{\bdry,\high,n} (\statevector_i^{\bdry,n} - \statevector_i^n).
\end{align*}
Here, we have used the fact that $\sum_{j \in \mathcal{I}(i)} (m_{ij} -
\delta_{ij} m_i) = 0$, which is a well known technique~\cite{DiscIndep2019,
kuzmin2012book} for absorbing the consistent mass matrix $m_{ij}$ into the
fluxes. We note that the high-order algebraic fluxes are skew symmetric,
$\bv{F}_{ij}^{\high} = - \bv{F}_{ji}^{\high}$. Furthermore, subtracting
\eqref{skewSymmetric} from \eqref{HighSchemeAlgebraic}, after some
reorganization we obtain:
\begin{align}
  \label{HighLowOrderIdentity}
  m_i \statevector_i^{\high, n+1} = m_i \statevector_i^{\low, n+1}
  + \sum_{j \in \mathcal{I}(i)} \bv{A}_{ij} + \bv{A}_{i}^{\bdry},
\end{align}
where $\bv{A}_{ij} := \bv{F}_{ij}^{\low} - \bv{F}_{ij}^{\high}$ and
$\bv{A}_{i}^{\bdry} := \bv{F}_{i}^{\bdry,\low} - \bv{F}_{i}^{\bdry,\high}$,
with $\bv{A}_{ij}$ skew symmetric, i.e. $\bv{A}_{ij}=-\bv{A}_{ij}$.
Equation~\ref{HighLowOrderIdentity} now serves as a starting point for the
convex limiting technique. We compute the new, blended update
$\statevector_i^{n+1}$ by setting
\begin{align}
  \label{LimitedScheme}
  m_i \statevector_i^{n+1} =
  m_i \statevector_i^{\low, n+1}
  + \sum_{j \in \mathcal{I}(i)} \limiter^n_{ij} \bv{A}_{ij} +
  \limiter_{i}^{\bdry,n} \bv{A}_{i}^{\bdry},
\end{align}
where $\limiter^n_{ij} = \limiter^n_{ji} \in [0,1]$ and
$\limiter_{i}^{\bdry,n} \in [0,1]$ are limiter coefficients. From
\eqref{LimitedScheme} and \eqref{HighLowOrderIdentity}, it is evident that
$\limiter_{ij}, \limiter_{i}^\bdry = 0$ recovers the low-order scheme and,
conversely, $\limiter_{ij}, \limiter_{i}^\bdry = 1$ the high-order scheme.
The goal is thus to choose the limiter coefficients as large as possible
while maintaining a pointwise invariant-set property (in the spirit of
Lemma~\ref{InvariantProp}), \ie, $\statevector_i^{n+1} \in \mathcal{A}$.

\begin{remark}[First-order scheme and high-order polynomials]
  In this work, we use the same finite element basis and stencil for both
  the high-order and low-order methods. However, we note that several
  authors \cite{Lohmann2017, Pazner2021, GuermondSubgrid} have explored the
  argument that the first-order method degrades their accuracy with
  increasing polynomial degree $k$. Therefore the first-order scheme should
  be computed using a subgrid of low-order polynomial degree, usuall $k =
  0$ or $k = 1$. While this is indeed true that the first-order method
  degrades its accuracy with increasing polynomial degree, the importance
  of such a degradation is not substantial for modest polynomial degree, as
  illustrated by Table \ref{ErrorGrowthTable}. For instance, in the table
  it can be observed that the error of the $\mathbb{Q}_2$ first-order
  scheme is $1.6\times $ the error of the $\mathbb{Q}_1$ first-order
  scheme. Similarly, the error of the $\mathbb{Q}_3$ first-order scheme is
  $2.3\times$ larger than the error of the corresponding $\mathbb{Q}_1$
  method using the same total number of DOFs. We believe that this higher
  error pre-factor is acceptable: in return we obtain a simpler method with
  a more straight-forward code, while avoiding all the complexity
  associated to have a low-order method defined in a subgrid. Of course,
  for very high order polynomial degrees, say $k \geq 5$, this sentiment
  might not hold true.
\end{remark}

%%%%%%%%%%%%%%%%%%%%%%%%%%%%%%%%%%%%%%%%%%%%%%%%%%%%%%%%%%%%%%%%%%%%%%%%%%%%%%%%
\subsection{Convex limiting: local bounds and line search}
\label{SubBoundsAndLim}

We want to enforce \emph{local} bounds on the density $\rho$ and the
specific entropy $s(\state)$ given by \eqref{EntropyEuler}. However, the
logarithm in \eqref{EntropyEuler} makes the specific entropy a rather
cumbersome quantity to work with directly. Following our previous work
\cite{Euler2018, DiscIndep2019, Guermond2022}, we use the rescaled
quantity $\widetilde{s}(\state) = \rho^{-\gamma}\inte(\state) =
\tfrac{1}{\gamma-1}\exp(s(\state))$ instead. Since it is a monotonic
rescaling, enforcing a minimum bound on $\widetilde{s}(\state)$ will also
enforce a minimum bound on $s(\state)$. For each node $i \in \vertices$, we
construct local bounds $\rhoimin$, $\rhoimax$, $\stildeimin$, and
construct limiter coefficients $\limiter^n_{ij}$, $\limiter_{i}^{\bdry,n}$,
such that the final update $\statevector_i^{n+1} = [\rho_i^{n+1},
\mom_i^{n+1}, \totme_i^{n+1}]$ given by \eqref{LimitedScheme},
satisfies bounds:
\begin{align*}
  \rhoimin \leq \rho_i^{n+1} \leq \rhoimax,
  \qquad
  \widetilde{s}(\statevector_i^{n+1}) \geq \stildeimin.
\end{align*}
In this manuscript we use the following local bounds
\begin{align}
  \label{eq:local_bounds}
  \begin{cases}
    \begin{aligned}
      \rhoimin &:= r_h^-\,
      \min\Big\{
        \rho(\overline{\statevector}_{i}^{\bdry,n})
        \,,\,
        \min_{j \in \mathcal{I}(i)}\min_{k \in \mathcal{I}(j)}\rho_k^n
        \,,\,
        \min_{j \in \mathcal{I}(i)}\min_{k \in \mathcal{I}(j)}
        \rho(\overline{\statevector}_{jk}^n)
      \Big\},
      \\
      \rhoimax &:= r_h^+\,
      \max\Big\{
        \rho(\overline{\statevector}_{i}^{\bdry,n})
        \,,\,
        \max_{j \in \mathcal{I}(i)}\max_{k \in \mathcal{I}(j)}\rho_k^n
        \,,\,
        \max_{j \in \mathcal{I}(i)}\max_{k \in \mathcal{I}(j)}
        \rho(\overline{\statevector}_{jk}^n)
      \Big\},
      \\
      \stildeimin &:= r_h^- \,
      \min\Big\{
        \widetilde{s}(\overline{\statevector}_{i}^{\bdry,n})
        \,,\,
        \min_{j \in \mathcal{I}(i)}\min_{k \in
        \mathcal{I}(j)}\widetilde{s}(\statevector_k^n)
        \,,\,
        \min_{j \in \mathcal{I}(i)}\min_{k \in
        \mathcal{I}(j)}\widetilde{s}(\overline{\statevector}_{jk}^n)
      \Big\},
    \end{aligned}
  \end{cases}
\end{align}
where the bar states $\overline{\statevector}_{jk}^n$ and
$\overline{\statevector}_{i}^{\bdry,n}$ are defined in \eqref{UsualBarState} and
\eqref{BoundaryBarState} respectively. We note that the use of the bar states in
\eqref{eq:local_bounds} is owed to equation \eqref{ConvexRef}: the
low-order update is a convex combination of the bar states; see also the
discussion in \cite[Section 4.1]{Euler2018} and \cite[Lemma
7.15]{DiscIndep2019}. The relaxation coefficients $r_h^\pm$ are defined as follows:
\begin{align*}
  r_h^{-} := 1 - c_r \widehat{h}_i^{p_r}, \quad
  r_h^{+} := 1 + c_r \widehat{h}_i^{p_r}, \quad
\end{align*}
where $\widehat{h}_i = \big(\tfrac{m_i}{|\Omega|} \big)^{\frac{1}{d}}$ is a
dimensionless mesh size. For the numerical tests in
Section~\ref{CompExperiments} we use the constants
\begin{align}
  \label{ViscMinConst2}
  c_r = 4.0 \ \ \text{and} \ \ p_r = 1.5
\end{align}
for all polynomial degrees throughout.
The constants have been chosen with a quick parametric study such that
we observe expected convergence rates for the numerical tests summarized
in Section~\ref{CompExperiments}.
The relaxation coefficients \eqref{ViscMinConst2} are necessary to recover
optimal convergence rates because a strict enforcement of the \emph{local}
minimum principle on the specific entropy would result in a first order
scheme; see \cite{Perthame1994, Euler2018}.
We note that this relaxation has no consequence on the robustness of the
scheme: As long as the initial data is admissible, the update will result
again in an admissible state.

A tempting alternative to the relaxation of the local bounds
\eqref{eq:local_bounds} is to dispense with using local bounds altogether
and replacing them with a single set of global bounds,
$\{\rho_{\text{min}}^{\text{global}}, \rho_{\text{max}}^{\text{global}},
\widetilde{s}_{\text{min}}^{\text{global}}\}$ with
$\rho_{\text{min}}^{\text{global}} > 0$ and
$\widetilde{s}_{\text{min}}^{\text{global}} > 0$, for all degrees of
freedom. While positivity preserving, such a limiter strategy--at least
from our experience---lacks control on over- and undershoots and leads to
unsatisfying numerical results for benchmark configurations.

The limiter coefficients $\limiter^n_{ij}$, $\limiter_{i}^{\bdry,n}$ are
now constructed with the help of one dimensional line searches
\cite{Euler2018}. For this, we first rewrite \eqref{eq:local_bounds} in
terms of a convex set,
\begin{align}
  \label{BiSet}
  \mathcal{B}_i = \Big\{
  \statevector = [\rho, \mom, \totme]^\transp \in
  \mathbb{R}^{d+2} \, \Big| \,
  \rhoimin \leq \rho \leq \rhoimax,\;
  \widetilde{s}(\statevector) \geq \stildeimin \Big\},
\end{align}
and rewrite \eqref{LimitedScheme} as follows:
\begin{align}
  \label{ConvexScheme}
  \statevector_i^{n+1}
  = \sum_{j \in \mathcal{I}(i)}
  \kappa_i \big(\statevector_i^{\low, n+1} + \ell^n_{ij} \bv{P}_{ij}\big)
  + \kappa_i \big(\statevector_i^{\low, n+1} + \ell_{i}^{\bdry,n}
  \bv{P}_{i}^{\bdry}\big),
\end{align}
with $\bv{P}_{ij} := \tfrac{1}{\kappa_i m_i}\bv{A}_{ij}$,
$\bv{P}_{i}^{\bdry} = \tfrac{1}{ \kappa_i m_i}\bv{A}_{i}^{\bdry}$, and
$\kappa_i := \big(\text{card}\mathcal{I}(i)+1\big)^{-1}$.

Equation \eqref{ConvexScheme} describes the update $\statevector_i^{n+1}$
as a convex combination of limited, unidirectional updates
$\statevector_i^{\low, n+1} + \ell^n_{ij} \bv{P}_{ij}$. This allows us to
reduce the construction of $\ell^n_{ij}$ and $\ell_{i}^{\bdry,n}$ to
solving one dimensional line searches \cite[Lemma 4.3]{Euler2018}:
\begin{lemma}
  \label{ConsInvariants}
  Assume that $\statevector_i^{\low, n+1} + \ell_{ij} \bv{P}_{ij} \in
  \mathcal{B}_i$ for all $j \in \mathcal{I}(i)$ and $\statevector_i^{\low,
  n+1} + \ell_{i}^{\bdry} \bv{P}_{i}^{\bdry} \in \mathcal{B}_i$. Then,
  $\statevector_i^{n+1}$ as defined by \eqref{ConvexScheme}, belongs
  to the set $\mathcal{B}_i$ as well.
  If $l_{ij}$ are symmetric, meaning $\ell_{ij} =
  \ell_{ji}$, then the convex limited high-order update
  $\statevector_i^{n+1}$ is also conservative.
\end{lemma}
In summary, the limiter coefficients are chosen such that
\begin{align*}
  \statevector_i^{\low, n+1} + \ell_{ij} \bv{P}_{ij} \in \mathcal{B}_i
  \text{ for all } j \in \mathcal{I}(i),
  \quad
  \text{and } \statevector_i^{\low, n+1} + \ell_{i}^{\bdry}
  \bv{P}_{i}^{\bdry} \in \mathcal{B}_i,
\end{align*}
which in turn implies that $\statevector_i^{n+1}\in \mathcal{B}_i$.

%%%%%%%%%%%%%%%%%%%%%%%%%%%%%%%%%%%%%%%%%%%%%%%%%%%%%%%%%%%%%%%%%%%%%%%%%%%%%%%%
%%%%%%%%%%%%%%%%%%%%%%%%%%%%%%%%%%%%%%%%%%%%%%%%%%%%%%%%%%%%%%%%%%%%%%%%%%%%%%%%
%%%%%%%%%%%%%%%%%%%%%%%%%%%%%%%%%%%%%%%%%%%%%%%%%%%%%%%%%%%%%%%%%%%%%%%%%%%%%%%%

\section{High-performance implementation and computational results}
\label{CompExperiments}

We now outline a high-performance implementation of the numerical scheme in
the hydrodynamic solver framework~\texttt{ryujin}~\cite{ryujin-2021-1,
Guermond2022}. The code supports discontinuous and continuous finite
elements on quadrangular meshes for the spatial approximation and is built
upon the \texttt{deal.II} finite element library~\cite{dealII95}. We
conclude the section by discussing a number of validation and benchmark
results.

\subsection{Implementation}

Due to the graph-based construction of the method, the implementation of the
proposed discontinuous Galerkin scheme can be realized in analogy to the
scheme described for continuous elements in~\cite{ryujin-2021-1}, and can be
applied to arbitrarily unstructured meshes including local adaptive refinement.
The implementation computes the necessary information row-by-row by a
(parallel) loop over the index range of variable $i$, with single-instruction
multiple-data (SIMD) vectorization across several rows to ensure a high
utilization of data-level parallelism. The kernels are written to balance
data access and computations for optimal performance on modern CPU-based
high-performance architecture, performing the following main steps:
\begin{itemize}
\item For the computation of the low-order update, the flux
  $\flux(\uvect^n_h)$ is evaluated point-wise and the graph
  viscosity~\eqref{eq:dij_low_order} is computed with a point-wise Riemann
  solver. For data locality reasons, most of the factors for the high-order
  viscosity are also computed in the necessary sweep over all mesh nodes.
\item The high-order update and convex-limiting steps~\eqref{ConvexScheme}
  involve combinations of the fluxes along the $i$ and $j$ indices as well as
  the high-order viscosity~\eqref{HighOrderViscosity}, together with the
  evaluation of the limiter coefficients $\ell^n_{ij}$.
\item In order to obtain converged results of the one-dimensional line
  searches, the limiter step is executed twice, necessitating two sweeps
  through all nodes.
\end{itemize}
As described in~\cite{ryujin-2021-1}, the computational cost of the above
steps is not only dominated by the actual flux computations, and the
elevated number of around 14--20 divisions per non-zero entry $(i, j)$ of
the stencil, but also by four transcendental power functions per non-zero
(using Pad\'e-type approximations) that play a crucial role. Furthermore,
the cost for indirect addressing into generic sparse matrix data structures
along the index $j$ are also relatively high. Overall, the proposed scheme
yields an arithmetic cost proportional to $\mathcal O((k+1)^d)$ operations
per degree of freedom. This is a substantial cost when compared to
state-of-the-art cell-based implementations of discontinuous Galerkin
methods, where modern implementations typically utilize on-the-fly
evaluation of the underlying finite-element integrals using
sum-factorization techniques for $\mathcal{O}(k+1)$ complexity per degree
of freedom~\cite{Kron2019,Fehn2019}, or related properties deduced by
spectral polynomial bases and one-dimensional differentiation
operations~\cite{Ranocha2023}. The cost has to be contrasted against the
mathematically proven robust realization proposed here.

Note that the chosen implementation does not utilize the additional
structure provided by the element-wise basis functions, which could allow
to fuse some of the indirect addressing and additional computations for
the unknowns inside a finite element cell, compared to the abstract row-by-row
processing of our approach. However, as the computational
part with expensive transcendental functions is the more restrictive
bottleneck on current architectures~\cite{ryujin-2021-1}, which is
addressed in ongoing research.

\subsection{Validation tests}

We now verify the proposed method for three different solution regimes: (i) a
smooth analytic solution given by the isentropic vortex~\cite{YeeSand1999};
(ii) a semi-smooth solution (continuous with second derivatives of bounded
variation) given by a single rarefaction wave; and (iii) the discontinuous
solution of the LeBlanc shock tube that has large pressure and density jumps.
For all three cases analytic expressions for the solution can be found in
\cite{Euler2018}, specifically we use \cite[Eq.\,5.3]{Euler2018},
\cite[Tab.\,2]{Euler2018}, and \cite[Tab.\,4]{Euler2018} with the same choices
for computational domains and parameters. As a figure of merit we introduce a
consolidated error norm:
\begin{align*}
  L^p\text{-error} \;:=\; \frac{\|\rho_h -
  \rho\|_{L^p(\Omega)}}{\|\rho\|_{L^p(\Omega)}}
  + \frac{\|\mom_h - \mom\|_{L^p(\Omega)}}{\|\mom\|_{L^p(\Omega)}}
  + \frac{\|\totme_h - \totme\|_{L^p(\Omega)}}{\|\totme\|_{L^p(\Omega)}},
\end{align*}
for $p = 1,2,\infty$ evaluated at the final time. We enforce Dirichlet
boundary conditions throughout by setting the boundary data
$\statevector_i^{\bdry, n}$ to the exact (time-dependent) solution. We describe
each test with more details in the following bullets:

\paragraph{Isentropic vortex (test case (i))}
We set the computational domain for the smooth test case (isentropic vortex) to
the square $[-5,5]^2$; see \cite[Eq.\,5.3]{Euler2018}. For the
sake of completeness we repeat the formulas of the exact analytical
solution \cite{YeeSand1999}:
\begin{align*}
\rho(\xcoord, t) = (\rho_\infty + \delta \rho(\xcoord, t))^{\frac{1}{\gamma
-1}} \ ,  \ \
\vel = \vel_\infty + \delta\vel \ ,  \ \
p(\xcoord, t) = \rho(\xcoord,t)^{\gamma} \\
\delta\vel(\xcoord,t) = \tfrac{\beta}{2 \pi} e^{1 - r^2}
[- \overline{\xcoord}_2 , \overline{\xcoord}_1]^\transp \ , \ \
\delta \rho(\xcoord,t) = - \tfrac{(\gamma - 1) \beta^2}{8\gamma \pi^2} e^{1 -
r^2}
\end{align*}
$\overline{\xcoord} = \xcoord - \xcoord_0 - t \vel_\infty$, $\xcoord =
[\xcoord_1, \xcoord_2]^\transp$ are the space coordinates, $\xcoord_0 =
[\xcoord_{10}, \xcoord_{20}]$ is the initial position of the vortex, and $r =
|\overline{\xcoord}|_{\ell^2}$. For all our tests: $\rho_\infty = 1$,
$\vel_{\infty} = [1,1]^\transp$, $\xcoord_0 = [-1, -1]$, $\gamma = 5/3$, and
$\beta = 5.0$. The initial time is $t_0 = 0$ and final time is $t_F = 2$.
We discretize the domain
with uniform grids with $n_e$ elements per edge, where $n_e = n_k \cdot 2^r$
with $n_k = 24, 16, 12$ for the cases polynomial degrees $k = 1,2,3$,
respectively. We now create a series of increasingly refined meshes by varying
$r$ from $0$ to $5$. The values of $n_k$ are chosen such that each refinement
level $r$ has the same number of degrees of freedom for all polynomial degrees
$k = 1,2,3$. For time integration we use
SSPRK-54 throughout, a fourth order strong stability preserving Runge Kutta
method~\cite{Spit2002}. Computational results are summarized in
Table~\ref{ErrorIsentropicConvex} in page~\pageref{ErrorIsentropicConvex}.
Classical error analysis for linear advection problems indicate that the
expected rate in the $L^2(\Omega)$-norm is of order $\mathcal{O}(h^{k +
\frac{1}{2}})$. In general, we observe expected convergence rates for all
reported test cases. A notable exception is a slight reduction
of convergence rates in the $L^1(\Omega)$ and $L^\infty(\Omega)$ norms for
polynomial degree $k=3$.

\paragraph{Rarefaction wave (test case (ii))}
Similarly, for the rarefaction test case (ii) we split the
unit interval $[0,1]$ into $n_e = n_k \cdot 2^r$ uniform subintervals with
$n_k = 60, 40, 30$ for $k = 1,2,3$, respectively, and by varying $r$ from
$0$ to $7$. Regarding time-integration, we use the SSPRK3 scheme
for all polynomial degrees. We note that the error for the rarefaction wave is
dominated by the fact that the initial data is non-differentiable in $\xcoord =
0.2$. This test also has the added difficulty that there is a sonic point at
$\xcoord = 0.2$: numerical methods without enough artificial viscosity will not
produce an entropic solution. Error estimates from polynomial interpolation
suggest a limit of $\mathcal{O}(h^2)$ for the convergence rate in the
$L^1$-norm. However, we are not aware of any scheme capable of delivering second
order rates for the rarefaction wave test. For instance, finite volume methods
with piecewise linear reconstructions deliver rates $\mathcal{O}(h^p)$ with $p
\in [1.333, 1.50]$, see \cite{PopovChua2021}; semi-discretely entropy-stable
methods yield $p \leq 1.50$ regardless of the polynomial degree, see
\cite{Chan2023}; first-order continuous finite elements with using
entropy-viscosity and convex limiting achieve $p \in [1.60, 1.65]$, see
\cite{Guermond2014}. Our results are reported in
Table~\ref{ErrorRarefactionConvex} on page~\pageref{ErrorRarefactionConvex}.
We observe a convergence order $\mathcal{O}(h^p)$ in the $L^1$-norm with
average $p\approx 1.70$, $1.60$, and $1.63$ for polynomial degrees $k=1$, $2$,
and $3$, respectively.

\paragraph{Leblanc shock tube (test case (iii))} Results for the LeBlanc
shocktube are summarized in Table~\ref{ErrorLeblancConvex}
page~\pageref{ErrorLeblancConvex}. We observe the expected linear convergence
rate for all polynomial degrees $k=1$, $2$, $3$. From mathematical approximation
theory, it is well known that high order polynomial degrees offer no advantage
when approximating discontinuous problems. In this sense, the numerical results
in Table~\ref{ErrorLeblancConvex} for the LeBlanc shocktube test are optimal.
Note that, in Table~\ref{ErrorLeblancConvex}, the exact same number of global
degrees of freedom is used on each refinement level for $\mathbb{Q}_1$,
$\mathbb{Q}_2$ and $\mathbb{Q}_3$ elements. Comparing the obtained $L^1$-error
for all three cases we tend to conclude that $\mathbb{Q}_1$ elements are the
optimal choice---at least for the case of discontinuous solutions with
strong shocks. For the same number of global degrees of freedom they offer
the smallest $L^1$-error while having low computational complexity, and
comparatively large time-step sizes.

\begin{remark}[Verification of boundary conditions with isentropic vortex]
  In order to test our implementation of boundary conditions, in Section
  \ref{subsec:BCvalidation} we modify the isentropic vortex benchmark. We
  increase the final time $t_F$ so that the simulation ends with the center
  of the isentropic situated exactly above the top right corner of the
  computational domain. By doing so, the correct treatment of boundary data
  is essential for recovering optimal convergence rates.
\end{remark}

%%%%%%%%%%%%%%%%%%%%%%%%%%%%%%%%%%%%%%%%%%%%%%%%%%%%%%%%%%%%%%%%%%%%%%%%%%%%%%%%
\subsection{Accuracy of boundary condition enforcement}
\label{subsec:BCvalidation}
We now briefly evaluate the performance of the boundary condition
enforcement in \eqref{cijdijscheme} and \eqref{HighScheme} by repeating
the smooth isentropic vortex test, case (i), with a modified final time $t_F$
for the convex-limited method and bilinear finite elements ($k=1$).
Specifically, we choose a final time of $t_F = \frac{6}{M}$ for increasing
choices of (directional) vortex speed $M = 1$, $1.5$, $2.0$, $2.5$. Here, $M=1$
implies that the center of the vortex is moving exactly with the speed of sound
$a$ in $x$ and in $y$ direction individually. We note that for $M = 1$ and $M =
1.5$ a significant portion of the top and right edges (outflow boundaries) will
be subsonic. With the choice of final time $t_F = \frac{6}{M}$, the center of
the vortex is located exactly on top of the top right corner at final time
$t_F$. As a rigorous figure of merit we examine convergence rates in the
$L^1$-norm. Four different strategies are tested for the construction of
boundary data $\statevector_i^{\bdry,n}$:
\begin{itemize}
  \item[(a)]
    exact Dirichlet data by setting $\statevector_i^{\bdry,n}
    :=\statevector_i^{\text{sol},n}$ on the entirety of the boundary, where
    $\statevector_i^{\text{sol},n}$ is the analytical solution;
  \item[(b)]
    sub/super-sonic boundary conditions with exact data:
    $\statevector_i^{\text{d},n}=\statevector_i^{\text{sol},n}$;
  \item[(c)]
    sub/super-sonic boundary conditions with a far-field state:
    $\statevector_i^{\text{d},n}=\statevector^{\text{far},n}$;
  \item[(d)]
    sub/super-sonic boundary conditions with the old state: we set
    $\statevector_i^{\text{d},n}=\statevector_i^{n}$.
\end{itemize}
Strategies (a) and (b) are intended to evaluate the formal consistency of
the method when the exact boundary data is available. On the other hand,
strategies (c) and (d) are meant to evaluate the performance of the method
when exact boundary data is not available. For strategies (c) and (d) no
rates can be expected as we evaluate the error up to the boundary. The
behaviour of strategies (c) and (d) are of particular interest in the
context of channel flows and transonic exterior aerodynamics. For such
applications exact boundary data is unavailable and the specific subsonic
or supersonic nature is unknown as well. The numerical results for the four
strategies are summarized in Table~\ref{ErrorBoundaryConditions} for the case
of $\mathbb{Q}_1$ spatial discretization. We see that strategies (a) and (b)
deliver the proper convergence rates. On the other hand, even though no rates
should be expected for the case of strategies (c) and (d), we still observe
proper convergence rates once the regime becomes fully supersonic ($M = 2$ and
$M = 2.5$). This indicates that the sub/super-sonic boundary condition approach
is indeed capable of selecting the boundary-data from the proper upwind
direction.

%%%%%%%%%%%%%%%%%%%%%%%%%%%%%%%%%%%%%%%%%%%%%%%%%%%%%%%%%%%%%%%%%%%%%%%%%%%%%%%%
\subsection{High fidelity simulation: Mach 3 flow past a cylinder}
We now present numerical results for a 2D benchmark configuration
consisting of a Mach 3 flow past a cylinder with with radius $0.25$ is
centered along $(0.6, 0, z)$. The computational domain is $\Omega = [0, 4]
\times [-1, 1]$ and is equipped with Dirichlet boundary conditions on the
left of the domain, slip boundary conditions on the cylinder and the top
and bottom of the domain, and do nothing boundary outflow conditions on the
right side. The initial flow configuration is that of a uniform flow at
Mach 3 \cite{Euler2018}. The computational domain is meshed with an
unstructured quadrilateral coarse mesh. A higher resolution is obtained by
subdividing every quadrilateral into 4 children an fixed number of times
and adjusting newly generated nodes on the cylinder boundary to lie on the
curved surface.
Figure~\ref{fig:cylinder} shows a temporal snapshot at time $t = 4.0$. The
computations where performed with a mesh consisting of 9.4M quadrilaterals
corresponding to 9.4M degrees of freedom per component for $\mathbb{Q}^1$,
and with a mesh consisting of 2.4M quadrilaterals for $\mathbb{Q}^2$ and
$\mathbb{Q}^3$, corresponding to 5.3M ($\mathbb{Q}^2$) and 9.4M
($\mathbb{Q}^3$) degrees of freedom per component.
We observe qualitatively that all spatial discretizations lead to a
comparable results with well captured (unstable) contact discontinuities
emerging from primary and secondary triple points.
\begin{figure}[p]
  \centering
  \subfloat[$\mathbb{Q}^1$, 9.4M DOFs per component]{\includegraphics[height=5.5cm]{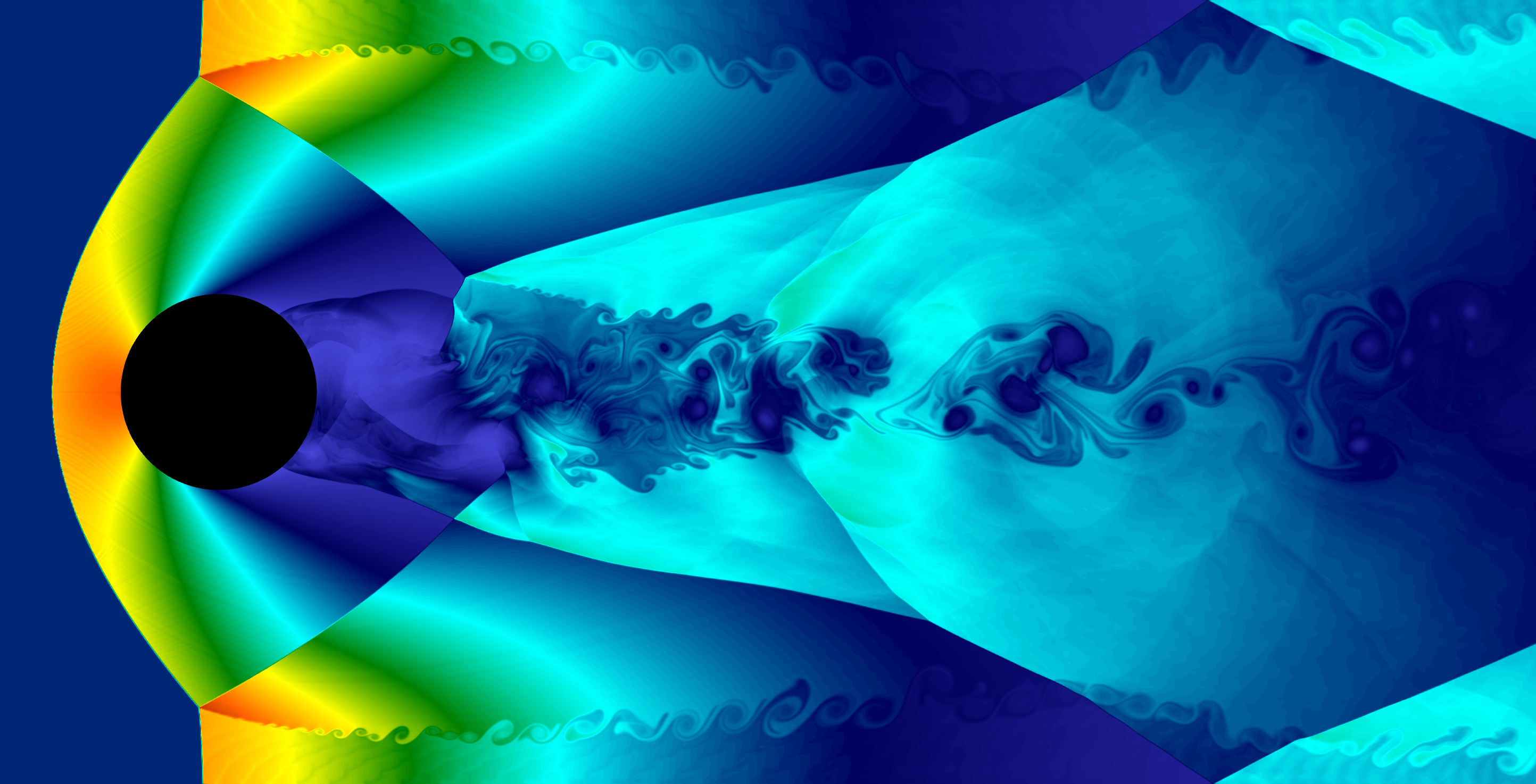}}

  \subfloat[$\mathbb{Q}^2$, 5.3M DOFs per component]{\includegraphics[height=5.5cm]{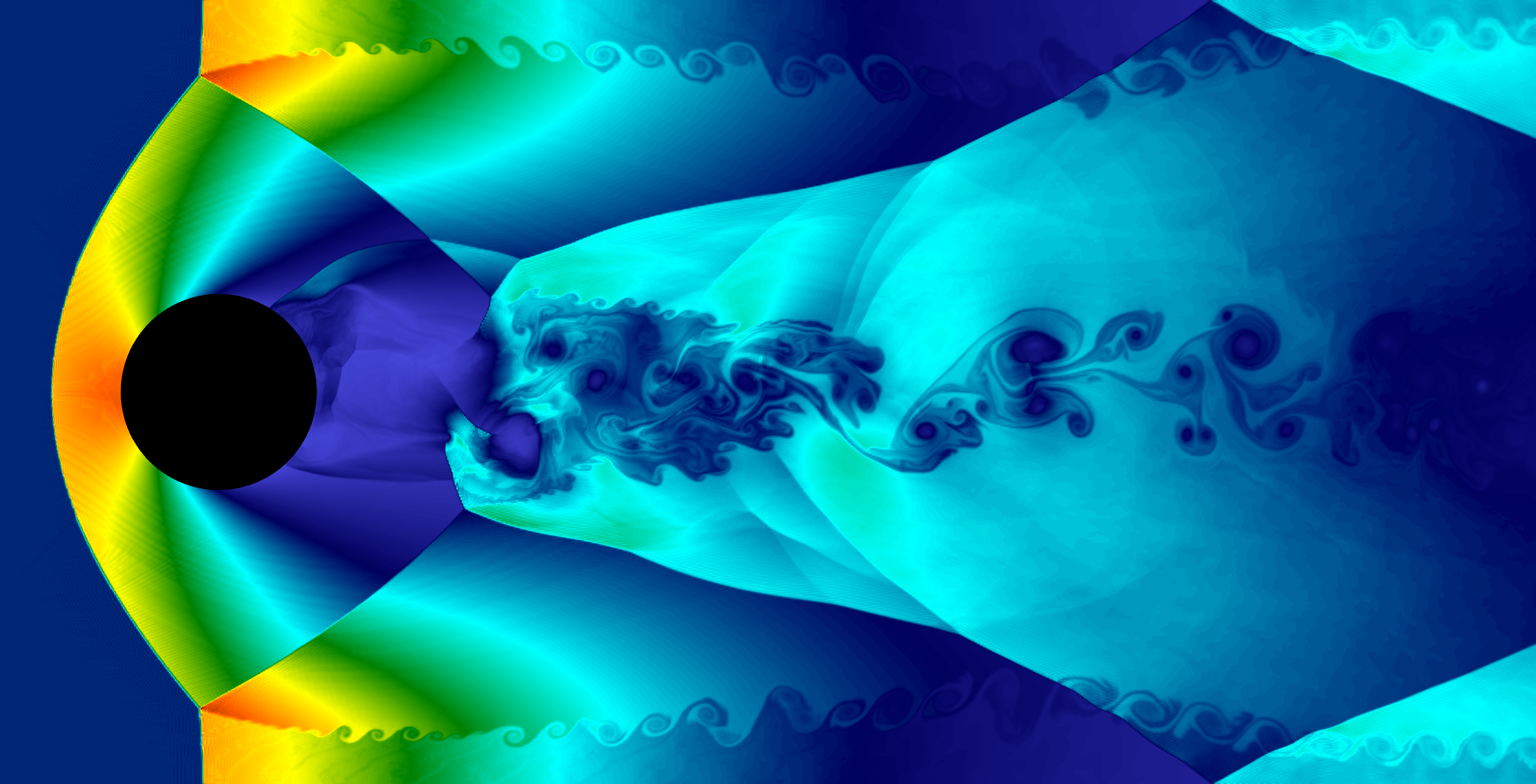}}

  \subfloat[$\mathbb{Q}^3$, 9.4M DOFs per component]{\includegraphics[height=5.5cm]{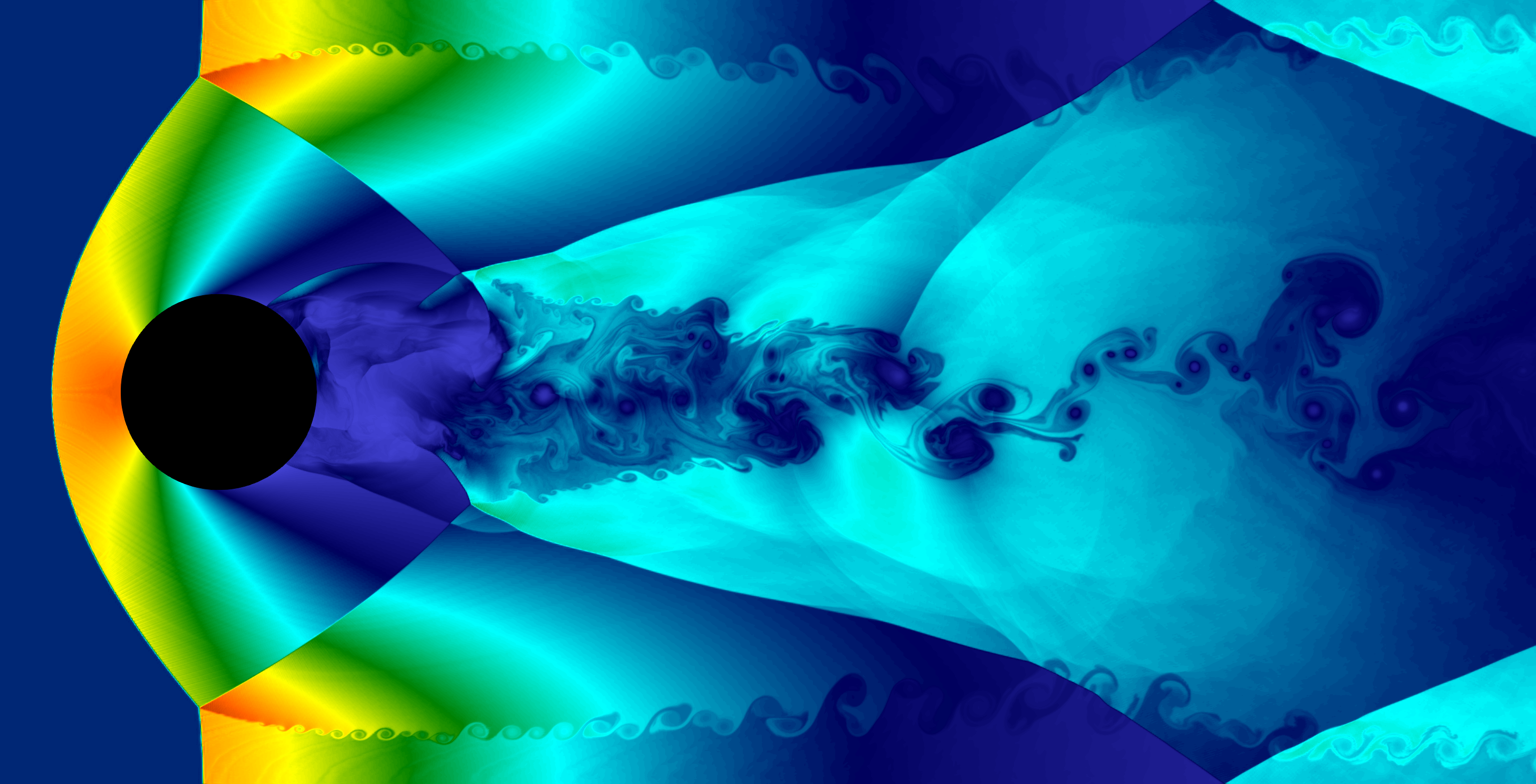}}
  \caption{Temporal snapshot at time $t=4.0$ of the density profiles of a
    supersonic Mach 3 flow past a cylinder. Computed for increasing
    polynomial degree: (a) $\mathbb{Q}^1$, (b) $\mathbb{Q}^2$, (c)
    $\mathbb{Q}^3$. The density is visualized on a rainbow colormap to
    highlight discontinuities.}
  \label{fig:cylinder}
\end{figure}

%%%%%%%%%%%%%%%%%%%%%%%%%%%%%%%%%%%%%%%%%%%%%%%%%%%%%%%%%%%%%%%%%%%%%%%%%%%%%%%%
%%%%%%%%%%%%%%%%%%%%%%%%%%%%%%%%%%%%%%%%%%%%%%%%%%%%%%%%%%%%%%%%%%%%%%%%%%%%%%%%
%%%%%%%%%%%%%%%%%%%%%%%%%%%%%%%%%%%%%%%%%%%%%%%%%%%%%%%%%%%%%%%%%%%%%%%%%%%%%%%%

\section{Conclusion and outlook}\label{conclusion}

We have introduced a graph-based discontinuous Galerkin method for solving
hyperbolic systems of conservation laws. The method has three main
ingredients: a first-order scheme, a high-order scheme (based on the
entropy-viscosity technique), and a convex-limiting procedure that blends
the high and low order schemes. The first-order update satisfies both the
invariant-domain property as well as a pointwise discrete entropy
inequality for any entropy of the system. The resulting convex-limited
scheme preserves the invariant set using relaxed local bounds.

A notable feature of the method is the direct incorporation of boundary
conditions. The state at each node is guaranteed to be admissible provided
the boundary data supplied to the scheme is admissible. For the case of the
first-order method, this allows to prove invariant-set preservation as well as
local entropy inequalities \emph{including} the effect of boundary contributions. For the
high-order and convex-limited scheme, we have tested the implementation
of boundary conditions using the isentropic vortex test with sufficiently large
final time, allowing interaction of the vortex with the boundary. If the
boundary data is the exact analytical solution, the method delivers optimal
convergence rates. On the other hand, if the boundary data consists of the
far-field state or the data from the previous time step, the implementation is
convergent in the fully supersonic regime.

The convex-limited scheme has been evaluated with a number of numerical
tests ranging from a smooth analytic solution to a discontinuous one,
observing expected convergence rates. The discontinuous test has verified
robustness of our scheme and first-order convergence in the
shock-hydrodynamics regime. Consistent with approximation theory, for the
same number of global degrees of freedom, the lowest order $\mathbb{Q}_1$
ansatz offers the smallest $L^1$-error while having the lowest
computational complexity, and comparatively large time-step sizes. Finally,
the semi-smooth rarefaction test has verified a rather subtle aspect of
high resolution methods, which is the ability to produce better than
first-order rates for solutions that are continuous with second derivatives
of bounded variation. Finally, we verified that the proposed method is
suitable for high fidelty simulations with a 2D benchmark configuration of
a Mach 3 flow past a cylinder.

%%%%%%%%%%%%%%%%%%%%%%%%%%%%%%%%%%%%%%%%%%%%%%%%%%%%%%%%%%%%%%%%%%%%%%%%%%%%%%%%
%%%%%%%%%%%%%%%%%%%%%%%%%%%%%%%%%%%%%%%%%%%%%%%%%%%%%%%%%%%%%%%%%%%%%%%%%%%%%%%%
%%%%%%%%%%%%%%%%%%%%%%%%%%%%%%%%%%%%%%%%%%%%%%%%%%%%%%%%%%%%%%%%%%%%%%%%%%%%%%%%

\section*{Acknowledgments}
This material is based upon work supported in part by the National Science
Foundation grant DMS-2045636 (MM), DMS-2409841 (IT), by the Air Force Office of
Scientific Research, USAF, under grant/contract number FA9550-23-1-0007 (MM),
and by the German Ministry of Education and Research through project ``PDExa:
Optimized software methods for solving partial differential equations on
exascale supercomputers'', grant agreement no. 16ME0637K (MK).

%%%%%%%%%%%%%%%%%%%%%%%%%%%%%%%%%%%%%%%%%%%%%%%%%%%%%%%%%%%%%%%%%%%%%%%%%%%%%%%%
%%%%%%%%%%%%%%%%%%%%%%%%%%%%%%%%%%%%%%%%%%%%%%%%%%%%%%%%%%%%%%%%%%%%%%%%%%%%%%%%
%%%%%%%%%%%%%%%%%%%%%%%%%%%%%%%%%%%%%%%%%%%%%%%%%%%%%%%%%%%%%%%%%%%%%%%%%%%%%%%%

\clearpage %TODO
\newpage
\appendix
\section{Convergence tables}
\phantom{text}

%%%%%%%%%%%%%%%%%%%%%%%%%%%%%%%%%%%%%%%%%%%%%%%%%%%%%%%%%%%%%%%%%%%%%%%%%%%%%%%%

\vspace{4cm}
\begin{table}[H]
  \begin{center}
  \text{$L^1\text{-error}$}\\[0.5em]
  \begin{tabular}{lcccccc}
    \toprule
    \text{\#DOFs} & $Q^1$ & \text{rate} & $Q^2$ & \text{rate} & $ Q^3$ & \text{rate}
    \\[0.3em]
    2304    & \qty{0.0225513  }{} &      & \qty{0.00978699 }{} &      & \qty{0.00335439 }{} &      \\
    9216    & \qty{0.00728095 }{} & 1.63 & \qty{0.00114886 }{} & 3.09 & \qty{0.000256794}{} & 3.71 \\
    36864   & \qty{0.00206252 }{} & 1.82 & \qty{0.000156191}{} & 2.88 & \qty{1.70566e-05}{} & 3.91 \\
    147456  & \qty{0.000548656}{} & 1.91 & \qty{2.0117e-05 }{} & 2.96 & \qty{1.14504e-06}{} & 3.9  \\
    589824  & \qty{0.000141453}{} & 1.96 & \qty{2.53867e-06}{} & 2.99 & \qty{7.78918e-08}{} & 3.88 \\
    2359296 & \qty{3.59386e-05}{} & 1.98 & \qty{3.19261e-07}{} & 2.99 & \qty{5.49642e-09}{} & 3.82 \\
    \bottomrule
  \end{tabular} \\
  \vspace{1em}
  \text{$L^2\text{-error}$}\\[0.5em]
  \begin{tabular}{lcccccc}
    \toprule
    \text{\#DOFs} & $Q^1$ & \text{rate} & $Q^2$ & \text{rate} & $ Q^3$ & \text{rate}
    \\[0.3em]
    2304    & \qty{0.0552303  }{} &      & \qty{0.0208165  }{} &      & \qty{0.00858812 }{} &      \\
    9216    & \qty{0.0183387  }{} & 1.59 & \qty{0.00285091 }{} & 2.87 & \qty{0.000684466}{} & 3.65 \\
    36864   & \qty{0.00533783 }{} & 1.78 & \qty{0.000429764}{} & 2.73 & \qty{4.53012e-05}{} & 3.92 \\
    147456  & \qty{0.00145487 }{} & 1.88 & \qty{6.09671e-05}{} & 2.82 & \qty{3.28015e-06}{} & 3.79 \\
    589824  & \qty{0.000382151}{} & 1.93 & \qty{8.42768e-06}{} & 2.85 & \qty{2.39001e-07}{} & 3.78 \\
    2359296 & \qty{9.82513e-05}{} & 1.96 & \qty{1.17759e-06}{} & 2.84 & \qty{1.92164e-08}{} & 3.64 \\
    \bottomrule
  \end{tabular} \\
  \vspace{1em}
  \text{$L^\infty\text{-error}$}\\[0.5em]
  \begin{tabular}{lcccccc}
    \toprule
    \text{\#DOFs} & $Q^1$ & \text{rate} & $Q^2$ & \text{rate} & $ Q^3$ & \text{rate}
    \\[0.3em]
    2304    & \qty{0.346729   }{} & --   & \qty{0.144562   }{} & --   & \qty{0.102838   }{} & --   \\
    9216    & \qty{0.12977    }{} & 1.42 & \qty{0.0353153  }{} & 2.03 & \qty{0.00651384 }{} & 3.98 \\
    36864   & \qty{0.046969   }{} & 1.47 & \qty{0.00930756 }{} & 1.92 & \qty{0.000546929}{} & 3.57 \\
    147456  & \qty{0.0169462  }{} & 1.47 & \qty{0.00182551 }{} & 2.35 & \qty{6.23241e-05}{} & 3.13 \\
    589824  & \qty{0.00561304 }{} & 1.59 & \qty{0.000371504}{} & 2.30 & \qty{5.25218e-06}{} & 3.57 \\
    2359296 & \qty{0.00175026 }{} & 1.68 & \qty{7.39064e-05}{} & 2.33 & \qty{5.76251e-07}{} & 3.19 \\
    \bottomrule
  \end{tabular}
  \end{center}
 \caption{\label{ErrorIsentropicConvex}\textbf{Convex-limited scheme:
   (i) isentropic vortex test.} Error delivered by scheme described by
   \eqref{LimitedScheme}; see Section~\ref{sec:highorder}. We
   consider the cases of $\mathbb{Q}^k$ spatial discretizations for $k =
   1,2,3$.}
\end{table}

%%%%%%%%%%%%%%%%%%%%%%%%%%%%%%%%%%%%%%%%%%%%%%%%%%%%%%%%%%%%%%%%%%%%%%%%%%%%%%%%
\begin{table}[p]
  \begin{center}
  \text{$L^1\text{-error}$}\\[0.5em]
  \begin{tabular}{lcccccc}
    \toprule
    \text{\#DOFs} & $Q^1$ & \text{rate} & $Q^2$ & \text{rate} & $ Q^3$ & \text{rate}
    \\[0.3em]
      120 & \qty{0.00177162 }{} & --   & \qty{0.0010682  }{} & --   & \qty{0.000300096}{} & --    \\
      240 & \qty{0.000517448}{} & 1.78 & \qty{0.000281316}{} & 1.92 & \qty{0.000102865}{} & 1.54  \\
      480 & \qty{0.000156424}{} & 1.73 & \qty{0.000112096}{} & 1.33 & \qty{3.45538e-05}{} & 1.57  \\
      960 & \qty{4.62784e-05}{} & 1.76 & \qty{3.54232e-05}{} & 1.66 & \qty{1.15774e-05}{} & 1.58  \\
     1920 & \qty{1.3013e-05 }{} & 1.83 & \qty{1.2685e-05 }{} & 1.48 & \qty{3.94556e-06}{} & 1.55  \\
     3840 & \qty{4.20382e-06}{} & 1.63 & \qty{3.8625e-06 }{} & 1.72 & \qty{1.04869e-06}{} & 1.91  \\
     7680 & \qty{1.36485e-06}{} & 1.62 & \qty{1.42314e-06}{} & 1.44 & \qty{3.57338e-07}{} & 1.55  \\
    15360 & \qty{4.51572e-07}{} & 1.6  & \qty{4.33493e-07}{} & 1.71 & \qty{1.06971e-07}{} & 1.74  \\
    \bottomrule
  \end{tabular} \\
  \end{center}
  \caption{\label{ErrorRarefactionConvex}\textbf{Convex-limited scheme:
    (ii) rarefaction test.} Convergence rates for scheme
    \eqref{LimitedScheme}. We consider the cases of $\mathbb{Q}^k$ spatial
    discretizations for $k = 1,2,3$. The number of elements for each case
    $k=1,2,3$ has been chosen in a way that the number of degrees of
    freedom of each refinement case match exactly. Note that the best
    expected rate for this test is $\mathcal{O}(h^2)$ for all polynomial
    degrees.}
\end{table}

%%%%%%%%%%%%%%%%%%%%%%%%%%%%%%%%%%%%%%%%%%%%%%%%%%%%%%%%%%%%%%%%%%%%%%%%%%%%%%%%
\begin{table}[p]
  \begin{center}
  \text{$L^1\text{-error}$}\\[0.5em]
  \begin{tabular}{lcccccc}
    \toprule
    \text{\#DOFs} & $Q^1$ & \text{rate} & $Q^2$ & \text{rate} & $ Q^3$ & \text{rate}
    \\[0.3em]
      120 & \qty{0.136198  }{} & --   & \qty{0.12563   }{} & --   & \qty{0.115759  }{} & --   \\
      240 & \qty{0.0951616 }{} & 0.52 & \qty{0.0686258 }{} & 0.87 & \qty{0.0595728 }{} & 0.96 \\
      480 & \qty{0.05519   }{} & 0.79 & \qty{0.0384026 }{} & 0.84 & \qty{0.0399522 }{} & 0.58 \\
      960 & \qty{0.0319462 }{} & 0.79 & \qty{0.0290277 }{} & 0.4  & \qty{0.0205211 }{} & 0.96 \\
     1920 & \qty{0.0163883 }{} & 0.96 & \qty{0.0144072 }{} & 1.01 & \qty{0.0121395 }{} & 0.76 \\
     3840 & \qty{0.00870375}{} & 0.91 & \qty{0.00849438}{} & 0.76 & \qty{0.00655946}{} & 0.89 \\
     7680 & \qty{0.00432392}{} & 1.01 & \qty{0.00473035}{} & 0.84 & \qty{0.004345  }{} & 0.59 \\
    15360 & \qty{0.0022895 }{} & 0.92 & \qty{0.00262618}{} & 0.85 & \qty{0.00229631}{} & 0.92 \\
    \bottomrule
  \end{tabular} \\
  \end{center}
  \caption{\label{ErrorLeblancConvex}\textbf{Convex-limited scheme:
    (iii) LeBlanc test.} Error delivered by scheme described by
    \eqref{LimitedScheme}; see Section~\ref{sec:highorder}. We
    consider the cases of $\mathbb{Q}^k$ spatial discretizations for $k =
    1,2,3$.}
\end{table}

%%%%%%%%%%%%%%%%%%%%%%%%%%%%%%%%%%%%%%%%%%%%%%%%%%%%%%%%%%%%%%%%%%%%%%%%%%%%%%%%
\clearpage %TODO
\newpage
\begin{table}
  \begin{center}
    \text{Strategy (a): exact Dirichlet data} \\
    \begin{tabular}{ccccccccc}
      \toprule
      \text{\#DOFs} & $M = 1.0$ & \text{rate}
                    & $M = 1.5$ & \text{rate}
                    & $M = 2.0$ & \text{rate}
                    & $M = 2.5$ & \text{rate}
      \\[0.3em]
      2304    & 4.37E-03 &  -   & 3.05E-03 &  -   & 3.05E-03 &   -  & 2.79E-03 &  -   \\
      9216    & 8.63E-04 & 2.34 & 6.26E-04 & 2.28 & 6.13E-04 & 2.31 & 5.68E-04 & 2.29 \\
      36864   & 1.78E-04 & 2.27 & 1.34E-04 & 2.21 & 1.26E-04 & 2.28 & 1.15E-04 & 2.29 \\
      147456  & 4.03E-05 & 2.14 & 3.15E-05 & 2.09 & 2.85E-05 & 2.14 & 2.60E-05 & 2.15 \\
      589824  & 9.61E-06 & 2.06 & 7.63E-06 & 2.04 & 6.79E-06 & 2.07 & 6.18E-06 & 2.07 \\
      2359296 & 2.45E-06 & 1.97 & 1.87E-06 & 2.02 & 1.65E-06 & 2.03 & 1.50E-06 & 2.03 \\
      \bottomrule
    \end{tabular}\\
    \vspace{1em}
    \text{Strategy (b): sub/super-sonic boundary conditions with exact Dirichlet data}
    \\[0.5em]
    \begin{tabular}{ccccccccc}
      \toprule
      \text{\#DOFs} & $M = 1.0$ & \text{rate}
                    & $M = 1.5$ & \text{rate}
                    & $M = 2.0$ & \text{rate}
                    & $M = 2.5$ & \text{rate}
      \\[0.3em]
      2304    & 4.05E-03 & -    & 2.98E-03 & -     & 2.41E-03 & -    & 1.99E-03 & -    \\
      9216    & 8.36E-04 & 2.27 & 6.19E-04 & 2.270 & 5.15E-04 & 2.22 & 4.29E-04 & 2.21 \\
      36864   & 1.79E-04 & 2.22 & 1.33E-04 & 2.210 & 1.13E-04 & 2.18 & 9.88E-05 & 2.12 \\
      147456  & 4.12E-05 & 2.12 & 3.13E-05 & 2.094 & 2.67E-05 & 2.08 & 2.37E-05 & 2.05 \\
      589824  & 9.92E-06 & 2.05 & 7.60E-06 & 2.044 & 6.52E-06 & 2.03 & 5.85E-06 & 2.02 \\
      2359296 & 2.54E-06 & 1.96 & 1.87E-06 & 2.024 & 1.60E-06 & 2.01 & 1.45E-06 & 2.01 \\
      \bottomrule
    \end{tabular}\\
    \vspace{1em}
    \text{Strategy (c): sub/super-sonic boundary conditions with far-field approximation}
    \\[0.5em]
    \begin{tabular}{ccccccccc}
      \toprule
      \text{\#DOFs} & $M = 1.0$ & \text{rate}
                    & $M = 1.5$ & \text{rate}
                    & $M = 2.0$ & \text{rate}
                    & $M = 2.5$ & \text{rate}
      \\[0.3em]
      2304    & 1.36E-02 & -    & 4.09E-03 & -   & 2.41E-03 & -    & 1.99E-03 & -    \\
      9216    & 1.07E-02 & 0.35 & 1.70E-03 & 1.26 & 5.15E-04 & 2.22 & 4.29E-04 & 2.21 \\
      36864   & 9.65E-03 & 0.14 & 1.14E-03 & 0.56 & 1.14E-04 & 2.17 & 9.88E-05 & 2.12 \\
      147456  & 9.20E-03 & 0.06 & 1.01E-03 & 0.17 & 2.73E-05 & 2.05 & 2.37E-05 & 2.05 \\
      589824  & 8.98E-03 & 0.03 & 9.31E-04 & 0.12 & 6.86E-06 & 1.99 & 5.85E-06 & 2.02 \\
      2359296 & 8.86E-03 & 0.02 & 8.72E-04 & 0.09 & 1.79E-06 & 1.93 & 1.45E-06 & 2.01 \\
      \bottomrule
    \end{tabular}\\
    \vspace{1em}
    \text{Strategy (d): sub/super-sonic boundary conditions with current state}
    \\[0.5em]
    \begin{tabular}{ccccccccc}
      \toprule
      \text{\#DOFs} & $M = 1.0$ & \text{rate}
                    & $M = 1.5$ & \text{rate}
                    & $M = 2.0$ & \text{rate}
                    & $M = 2.5$ & \text{rate}
      \\[0.3em]
      2304    & 6.12E-03 &  -    & 3.02E-03 &  -   & 2.41E-03 &   -  & 2.03E-03 &  -   \\
      9216    & 3.99E-03 & 0.61  & 6.58E-04 & 2.20 & 5.08E-04 & 2.24 & 4.33E-04 & 2.22 \\
      36864   & 3.59E-03 & 0.15  & 1.76E-04 & 1.90 & 1.13E-04 & 2.16 & 9.88E-05 & 2.13 \\
      147456  & 1.05E-02 & -1.56 & 7.09E-05 & 1.31 & 2.67E-05 & 2.08 & 2.37E-05 & 2.05 \\
      589824  & 3.97E-03 & 1.41  & 4.45E-05 & 0.67 & 6.52E-06 & 2.03 & 5.85E-06 & 2.02 \\
      2359296 & 4.66E-03 & -0.22 & 3.64E-05 & 0.29 & 1.60E-06 & 2.01 & 1.45E-06 & 2.01 \\
      \bottomrule
    \end{tabular}
  \end{center}
  \caption{\label{ErrorBoundaryConditions}
    \textbf{Boundary condition validation: modified isentropic vortex
    test.} We use a modified isentropic vortex test to verify the accuracy
    of the boundary value enforcement outlined in
    Section~\ref{sec:boundary_conditions}. We test four different
    strategies for the construction of boundary data: (a) exact Dirichlet
    data, (b) sub/super-sonic boundaries with exact analytical data, (c)
    sub/super-sonic boundaries with far-field state; and (d)
    sub/super-sonic boundaries using data from the previous time-step; see
    Section \ref{subsec:BCvalidation}. The test is repeated for different
    Mach numbers, $M=1$, $1.5$, $2.0$, and $2.5$. In every case
    we use $\mathbb{Q}_1$ spatial discretization
    We note that for $M = 1$ and $M = 1.5$ a significant portion of the
    outflow boundary (top and right edges of the domain) is subsonic.
    Strategies (a) and (b) deliver proper convergence rates. Strategies (c)
    and (d) are only convergent for sufficiently large Mach numbers.}
\end{table}

\begin{table}[H]
\begin{center}
 \begin{tabular}{cccc}
  \toprule
  \text{\#DOFs} & $\mathbb{Q}_1$ & $\mathbb{Q}_2$ &
  $\mathbb{Q}_3$
        \\[0.3em]
        \cmidrule(lr){1-4}
        2304   & 1.2799e-01 & 1.5639e-01 & 1.7400e-01 \\
        9216   & 8.5427e-02 & 1.1676e-01 & 1.4080e-01 \\
        36864  & 5.1002e-02 & 7.6054e-02 & 9.8985e-02 \\
        147456 & 2.8281e-02 & 4.4705e-02 & 6.1502e-02 \\
        589824 & 1.4981e-02 & 2.4553e-02 & 3.5023e-02 \\
        \bottomrule
  \end{tabular}

  \caption{\textbf{Error of the first-order method with respect polynomial
degree.}\label{ErrorGrowthTable} This tables illustrates the growth of the
$L^1$-error of the first-order scheme
\eqref{cijdijscheme}-\eqref{eq:dij_low_order} as the polynomial degree grows
for the isentropic vortex problem. For every polynomial degree the error halves
with each mesh refinement (as expected from a first-order scheme). However, the
error constant, or ``error pre-factor'', grows with the polynomial
degree. }
\end{center}
\end{table}

%%%%%%%%%%%%%%%%%%%%%%%%%%%%%%%%%%%%%%%%%%%%%%%%%%%%%%%%%%%%%%%%%%%%%%%%%%%%%%%%
%%%%%%%%%%%%%%%%%%%%%%%%%%%%%%%%%%%%%%%%%%%%%%%%%%%%%%%%%%%%%%%%%%%%%%%%%%%%%%%%
%%%%%%%%%%%%%%%%%%%%%%%%%%%%%%%%%%%%%%%%%%%%%%%%%%%%%%%%%%%%%%%%%%%%%%%%%%%%%%%%

\clearpage
\newpage
\section{Implementation of boundary conditions: adaptation to continuous finite
elements}\label{AppCg}

The scheme \eqref{cijdijscheme}--\eqref{eq:dij_low_order} and the
mathematical theory developed in this manuscript is entirely valid for the
case of continuous finite elements. However, slight changes are required in
definitions of \eqref{eq:cij_def1}--\eqref{eq:cij_def2}. Here provide the
required mathematical statements but avoid doing the proofs since they are
just adaptations of the the proof already presented for the discontinuous
case in Section \ref{sec:LowOrder}.

In the context of continuous finite elements we have to use the following
definitions:
\begin{align}
  \label{cijcidefCg}
  \bv{c}_{ij} := \int_{\Omega} \nabla\HypBasisScal_j \HypBasisScal_i \dx
  - \tfrac{1}{2} \int_{\partial \Omega} \HypBasisScal_{j} \HypBasisScal_i
  \normal_{\bdry} \ds, \qquad
  \bv{c}_i^{\bdry} = \tfrac{1}{2} \int_{\partial \Omega} \HypBasisScal_i
  \normal_{\bdry} \ds \, .
\end{align}

Note that the face integral $\tfrac{1}{2} \int_{\partial \Omega}
\HypBasisScal_{j} \HypBasisScal_i \normal_{\bdry} \ds$ can only be nonzero
if both $\HypBasisScal_{j}$ and $\HypBasisScal_i$ have support on the
boundary.

\begin{proposition}
  The vectors $\bv{c}_{ij}$ as defined in \eqref{cijcidefCg} satisfy the
  usual skew-symmetry property $\bv{c}_{ij} = - \bv{c}_{ji}$ for all $i,j
  \in \vertices$
\end{proposition}

\begin{proof}
  The proof follows by integration by parts arguments. However, in this
  case, since the shape functions $\{\HypBasisScal_i\}_{i \in \vertices}$
  are compactly supported and weakly differentiable in $\Omega$, we do not
  need to use integration by parts on each element $\element$, but rather
  integration by parts in entire domain $\Omega$, see also
  \cite{DiscIndep2019}.
\end{proof}

\begin{proposition}[Partition of unity properties] The vectors $\bv{c}_{ij}$ and
$\bv{c}_{i}^{\bdry}$as defined in \eqref{cijcidefCg} satisfy the
partition of unity property $\sum_{j \in \mathcal{I}(i)} \bv{c}_{ij} +
\bv{c}_{i}^{\bdry} = \bzero$.
\end{proposition}

The proof follows using arguments similar to those of Lemma \ref{LemPart}.

\begin{lemma}[Total balance] The scheme \eqref{cijdijscheme} with
$\bv{c}_{ij}$ and $\bv{c}_i^{\bdry}$ as defined in \eqref{cijcidefCg}
satisfies the flux-balance \eqref{totalBalance}
\end{lemma}

The proof of this lemma is omitted since it is identical to the proof of Lemma
\ref{TotBalLemma}.

%%%%%%%%%%%%%%%%%%%%%%%%%%%%%%%%%%%%%%%%%%%%%%%%%%%%%%%%%%%%%%%%%%%%%%%%%%%%%%%%
%%%%%%%%%%%%%%%%%%%%%%%%%%%%%%%%%%%%%%%%%%%%%%%%%%%%%%%%%%%%%%%%%%%%%%%%%%%%%%%%
%%%%%%%%%%%%%%%%%%%%%%%%%%%%%%%%%%%%%%%%%%%%%%%%%%%%%%%%%%%%%%%%%%%%%%%%%%%%%%%%

\bibliographystyle{siamplain}

\end{document}